\documentclass[11pt,a4paper]{article}

\usepackage[latin1]{inputenc}
\usepackage{enumerate}
\usepackage{mathrsfs}
\usepackage{color}
\usepackage[pdftex]{graphicx}
\usepackage{amssymb}
\usepackage{amsmath,amsthm}
\usepackage{hyperref}
\usepackage{soul}
\hypersetup{
  colorlinks=false,
  plainpages=false,
}

\DeclareGraphicsExtensions{pdf, jpg, jpeg, png}

\graphicspath{{.}}

\newtheorem{theorem}{Theorem}[section]
\newtheorem{lemma}[theorem]{Lemma}
\newtheorem{proposition}[theorem]{Proposition}
\newtheorem{corollary}[theorem]{Corollary}

\newtheorem{definition}[theorem]{Definition}

\theoremstyle{remark}
\newtheorem*{remark}{Remark}
\newtheorem*{remarks}{Remarks}

\numberwithin{equation}{section}
\numberwithin{figure}{section}

\newcommand{\R}{\mathbb{R}}

\newcommand{\Z}{\mathbb{Z}}

\renewcommand{\L}{\mathscr{L}}
\newcommand{\LL}{\mathbb{L}}
\renewcommand{\S}{\mathbb{S}}

\newcommand{\RI}{\what{\mathbb{R}}}
\newcommand{\RP}{\mathbb{R}\textup{P}}

\newcommand{\cycplanes}{\mathcal{P}^{4,2}_{cyc}}

\renewcommand{\P}{\operatorname{P}}
\newcommand{\pol}{\operatorname{pol}}

\newcommand{\inc}{\operatorname{inc}}

\newcommand{\vspan}{\operatorname{span}}
\renewcommand{\phi}{\varphi}
\renewcommand{\epsilon}{\varepsilon}
\newcommand{\longto}{\longrightarrow}
\newcommand{\wtilde}{\widetilde}
\newcommand{\what}{\widehat}

\renewcommand{\implies}{\Rightarrow}

\newcommand{\e}{\mathbf e}

\title{Curvature line parametrized surfaces and
    orthogonal coordinate systems.\\
		Discretization with Dupin cyclides.\thanks{Partially supported by the DFG
		Research Unit ``Polyhedral Surfaces'' and the DFG Research Center
		\textsc{Matheon}}
    }
\author{
Alexander I. Bobenko\footnote{
TU Berlin.
E-mail: bobenko@math.tu-berlin.de
}\hspace{3pt} and Emanuel Huhnen-Venedey\footnote{
TU Berlin.
E-mail: huhnen@math.tu-berlin.de
}
}

\begin{document}

\maketitle

\begin{abstract}
Cyclidic nets are introduced as discrete analogs of curvature line parametrized
surfaces and orthogonal coordinate systems. A 2-di\-men\-sional cyclidic net is
a piecewise smooth $C^1$-surface built from surface patches of Dupin cyclides,
each patch being bounded by curvature lines of the supporting cyclide. An
explicit description of cyclidic nets is given and their relation to the
established discretizations of curvature line parametrized surfaces as circular,
conical, and principal contact element nets is explained.  We introduce
3-dimensional cyclidic nets as discrete analogs of triply-orthogonal coordinate
systems and investigate them in detail. Our considerations are based on the Lie
geometric description of Dupin cyclides. Explicit formulas are derived and
implemented in a computer program.
\end{abstract}

%\begin{center}
%The final publication is available at link.springer.com\\
%(Geometriae Dedicata)\\
%\texttt{http://dx.doi.org/10.1007/s10711-011-9653-5}
%\end{center}

\section{Introduction}
Discrete differential geometry aims at the development of discrete
equivalents of notions and methods of classical differential
geometry. The present paper deals with \emph{cyclidic nets}, which
appear as discrete analogs of curvature line parametrized surfaces
and orthogonal coordinate systems.

A 2-dimensional cyclidic net in $\R^3$ is a piecewise smooth
$C^1$-surface as shown in Fig.~\ref{fig:cyclidic_from_circular}.
\begin{figure}[htb]
\centering
\parbox{4.5cm}{\includegraphics[scale=.15]{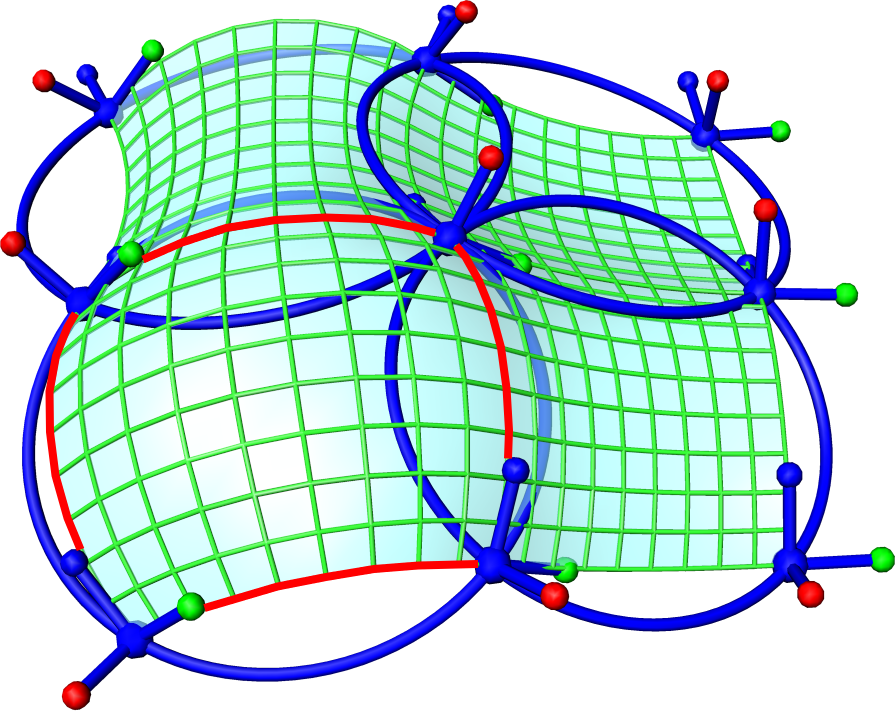}}
\quad \quad $\leftrightarrow$ \quad \quad
\parbox{4.5cm}{\includegraphics[scale=.14]{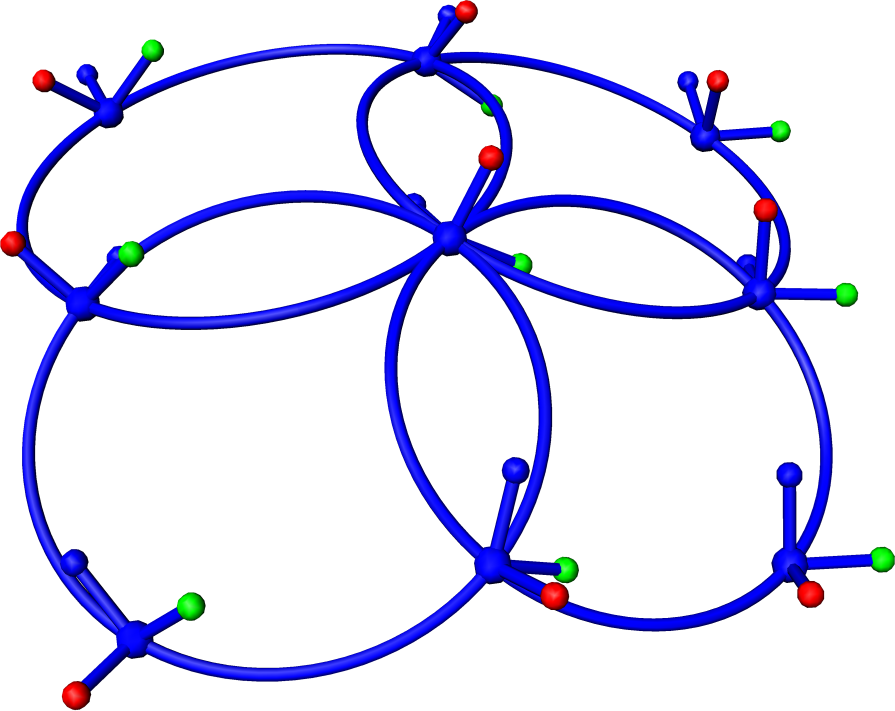}}
\caption{A 2-dimensional cyclidic net is a piecewise smooth
$C^1$-surface composed of cyclidic patches. Equivalently, it can
be seen as a circular net with frames at vertices.}
\label{fig:cyclidic_from_circular}
\end{figure}
Such a net is built from {\em cyclidic patches}. The latter are
surface patches of Dupin cyclides cut out along curvature lines.
Dupin cyclides themselves are surfaces in $\R^3$, which are
characterized by the property that their curvature lines are
circles. So curvature lines of the cyclidic patches in a 2D
cyclidic net constitute a net of $C^1$-curves composed of circular
arcs, which can be seen as curvature lines of the cyclidic net.

Discretization of curvature line parametrized surfaces and
orthogonal coordinate systems is an important topic of current
research in discrete differential geometry. The most recognized
discretizations (see \cite{BobenkoSuris:2008:DDGBook}) include:
\begin{itemize}
\item {\em circular nets}, that are quadrilateral nets with circular
quadrilaterals,
\item {\em conical nets}, that are quadrilateral nets, such that
four adjacent quadrilaterals touch a common cone of
revolution (or a sphere),
\item {\em principal contact element nets}, that are nets built from
planes containing a distinguished vertex (such a pair being called a \emph{contact element}), such that the vertices
constitute a circular net and the planes form a conical net.
\end{itemize}
The most elaborated discretization of curvature line pa\-ra\-metrized surfaces
are circular nets. They appear indirectly and probably for the first time in
\cite{Martin:1983:PrincipalPatches}. The circular discretization
of triply orthogonal coordinate systems was suggested in 1996 and
published in \cite{Bobenko:1999:DiscreteConformalMaps}. The next crucial step in the
development of the theory was made in
\cite{CieslinskiDoliwaSantini:1997:CircularNets}, where circular nets were
considered as a reduction of quadrilateral nets and were
generalized to arbitrary dimension. An analytic description of
circular nets in terms of discrete Darboux systems was given in
\cite{KonopelchenkoSchief:1998:3DIntegrableLattices},
and a Clifford algebra description of circular nets can be found in
\cite{BobenkoHertrich-Jeromin:2001:O-netsClifford}.
The latter was then used to prove the
$C^\infty$-convergence of circular nets to the corresponding smooth curvature line
parametrized surfaces (and orthogonal coordinate systems) in
\cite{BobenkoMatthesSuris:2003:OrthogonalSystems}.
Conical nets were introduced in
\cite{LiuPottmannWallnerYangWang:2006:ConicalNets} as quadrilateral nets
admitting parallel face offsets. This property was interpreted as
a discrete analog of curvature line parametrization. Circular and
conical nets were unified under the notion of principal contact
element nets in \cite{BobenkoSuris:2007:OrganizingPrinciples}
(see also \cite{PottmannWallner:2007:FocalGeometry}).
The latter belong to Lie sphere geometry.

All these previous discretizations are contained in cyclidic nets.  The vertices
of a cyclidic patch are concircular, therefore the vertices of a cyclidic net
build a circular net. On the other hand the tangent planes at these vertices
form a conical net, so vertices and tangent planes together constitute a
principal contact element net.  A cyclidic net contains more information than
its corresponding principal contact element net: The latter can be seen as a
circular net with normals (to the tangent planes) at vertices.  At each vertex
of a cyclidic net one has not only this normal, but a whole orthonormal frame.
The two additional vectors are tangent to the boundaries of the adjacent
patches and the patches are determined by this condition
(cf.  Fig.~\ref{fig:cyclidic_from_circular}).
We explain in detail how a cyclidic patch is determined already by its
vertices and one frame and derive its curvature line parametrization from that
data. It turns out that a cyclidic net is uniquely determined by the
corresponding circular net and the frame at one vertex. It is worth to mention
that the involved frames played an essential role in the convergence proof in
\cite{BobenkoMatthesSuris:2003:OrthogonalSystems}.

Composite $C^1$-surfaces built from cyclidic patches have been considered in
Computer Aided Geometric Design (CAGD), see, e.g.,
\cite{Martin:1983:PrincipalPatches,McLean:1985:CyclideSurfaces,MartinDePontSharrock:1986:CyclideSurfaces,DuttaMartinPratt:1993:CyclideSurfaceModeling,SrinivasKumarDutta:1996:SurfaceDesignUsingCyclidePatches}.
A large part of the book \cite{NutbourneMartin:1988:SurfaceDesign} by Nutbourne
and Martin deals with the Euclidean description of ``principal patches'' and in
particular cyclidic patches. It also contains parametrization formulas for
cyclidic patches, determined by a frame at one of the four concircular vertices.
Unfortunately these formulas are rather complicated and turned out to be not
appropriate for our purposes.  A sequel, focused mainly on surface composition
using cyclidic patches, was planned. However, the second book never appeared,
probably due to the limited possibilities of shape control.  Another application
of Dupin cyclides in CAGD is blending between different surfaces like cones and
cylinders, which are special cases of Dupin cyclides (see for example
\cite{Degen:2002:BlendingWithCyclides}, which also contains a short analysis of
Dupin cyclides in terms of Lie geometry).

In contrast to the previous works on surfaces composed of cyclidic patches, the
present work relies essentially on Lie geometry. We use the elegant classical
description of Dupin cyclides by Klein and Blaschke
\cite{Klein:1926:GeometrieVorlesungen,Blaschke:1929:DG3} to prove statements
about cyclidic patches, and to derive their curvature line parametrization. On
the other hand, cyclidic nets as well as circular nets are objects of M\"obius
geometry, and we argue several times from a M\"obius geometric viewpoint. In
Appendix~\ref{sec:lie_geometry} we give a brief overview of the projective model
of Lie geometry.

Another central part of the present work is a new discretization of
orthogonal coordinate systems as cyclidic nets (in particular in $\R^3$) .
Smooth orthogonal coordinate systems are important examples of
integrable systems \cite{Zakharov:1998:IntegrationLame}.
In the discrete setting analytic methods of the
theory of integrable systems have been applied already to circular nets in
\cite{DoliwaManakovSantini:1998:mDCircular} ($\bar\partial$-method) and in
\cite{AkhmetshinVolvovskiKrichever:1999:DiscreteDarboux-EgorovMetrics} (algebro-geometric
solutions).
Our approach to the discretization is the following:
In the smooth case the classical Dupin theorem states that in a
triply orthogonal coordinate system of $\R^3$ all coordinate surfaces
intersect along common curvature lines.
This translates to a 3D cyclidic net with its
2D layers seen as discrete coordinate surfaces,
for which orthogonal intersection at common vertices is required.
It turns out that this implies orthogonal intersection of
the cyclidic coordinate surfaces along common discrete coordinate lines (cf. Fig.~\ref{fig:octrihedron}).
Moreover it is shown that generically the discrete families of coordinate surfaces
can be extended to piecewise smooth families of 2D cyclidic nets,
which gives orthogonal coordinates for an open subset of $\R^3$.
The corresponding theory is extended to higher dimensions.
\begin{figure}[htb]
\centering
\includegraphics[scale=.33]{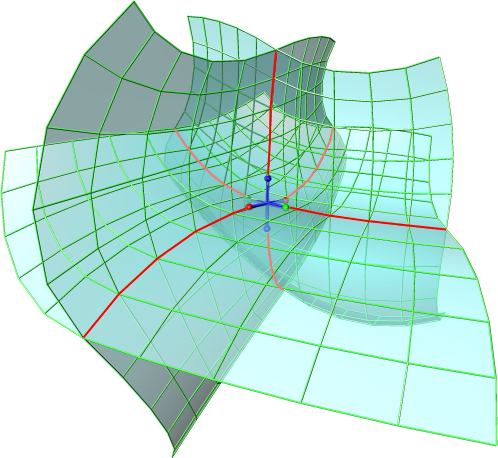}
\caption{2D layers of a 3D cyclidic net intersect orthogonally in
discrete coordinate lines, and in particular at vertices.}
\label{fig:octrihedron}
\end{figure}

Our construction of cyclidic nets, based on the projective model of Lie geometry,
has been implemented in a Java webstart\footnote{available at
{\tt www3.math.tu-berlin/geometrie/ps/software.shtml}}.
All 3D pictures of cyclidic nets have been created with this program.
A brief description of the implementation is given
in Section~\ref{subsec:implementation}.

\enlargethispage{\baselineskip}

\paragraph{Notations.}
We denote the compactified Euclidean space
\begin{equation*}
\RI^N := \R^N \cup \left\{ \infty \right\} \cong \S^N.
\end{equation*}
Homogeneous coordinates of a projective space are also marked with a hat, i.e.
$x=[\what x]=\P(\R \what x)$.
The span of projective subspaces $U_i = \P(\what U_i),i=1,\dots,n$ we write as
\begin{equation*}
\inc[U_1,\dots,U_n] := \P(\vspan(\what U_1,\dots,\what U_n)).
\end{equation*}
The corresponding polar subspace with respect to a quadric is denoted
\begin{equation*}
\pol[U_1,\dots,U_n] := \pol[\inc[U_1,\dots,U_n]] = \P(\vspan(\what U_1,\dots,\what U_n)^\perp),
\end{equation*}
where $\perp$ denotes the orthogonal complement with respect to the inner product defining the quadric.

\emph{Pay attention to the different use of $\langle \cdot,\cdot \rangle$} in our work.
We denote this way the inner product on $\R^{N+1,2}$ of signature $(N+1,2)$ which defines the Lie quadric
in the projective space $\RP^{N+1,2} = \P(\R^{N+1,2})$
(cf. Appendix~\ref{sec:lie_geometry}),
as well as the standard Euclidean scalar product in $\R^N$.
The products can be distinguished by the notation of the involved vectors:
As homogeneous coordinates of $\RP^{N+1,2}$ the vectors in $\R^{N+1,2}$ carry a hat,
while vectors in $\R^N$ don't.
Thus $\langle \what x_1,\what x_2 \rangle$ denotes the product in $\R^{N+1,2}$
and $\langle x_1,x_2 \rangle$ denotes the Euclidean scalar product in $\R^N$.

\section{Dupin cyclides}
\label{sec:dupin_cyclides}

\emph{Dupin cyclides are surfaces
which are characterized by the property that all curvature lines are generalized circles},
i.e. proper circles or straight lines.
They are invariant under Lie transformations,
and in this paper we use the projective model of Lie geometry for their analytic description
\cite{Klein:1926:GeometrieVorlesungen,Blaschke:1929:DG3,Pinkall:1985:DupinHypersurfaces}.

\subsection{Description in the projective model of Lie geometry}
\label{subsec:dupin_cyclides_definition}

In the classical book \cite{Blaschke:1929:DG3} by Blaschke one finds

\begin{definition}[Dupin cyclide]
A \emph{Dupin cyclide} in $\R^3$ is the simultaneous envelope of
two 1-parameter families of spheres, where each sphere out of one
family is touched by each sphere out of the complementary family.
\label{def:dupin_cyclide_euclidean}
\end{definition}

From the Euclidean point of view there are 6 different types of Dupin cyclides.
A Dupin cyclide is called to be of \emph{ring}, \emph{horn}, or \emph{spindle}
type depending on whether it has 0, 1 or 2 singular points.  Moreover one
distinguishes bounded and unbounded cyclides, where the latter are called
\emph{parabolic}.  The standard tori are Dupin cyclides (cf.
Fig.~\ref{fig:standard_tori}) and each other cyclide can be obtained from a
standard torus by inversion in a sphere \cite{Pinkall:1986:DupinscheZykliden}.
In particular, the parabolic cyclides are obtained if the center of the sphere
is a point of the cyclide (a parabolic horn cyclide with its singular point at
infinity is a circular cylinder, a parabolic spindle cyclide with one of its
singular points at infinity is a circular cone).

\begin{figure}[htb]
\centering
\parbox{4.15cm}{\includegraphics[scale=.25]{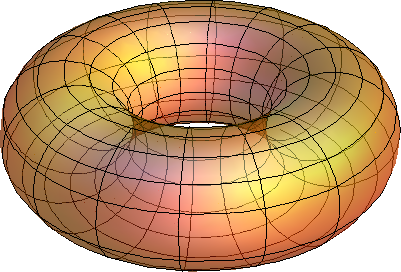}}
\parbox{4cm}{\includegraphics[scale=.24]{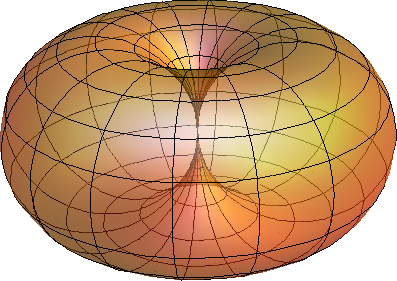}}
\parbox{3cm}{\includegraphics[scale=.21]{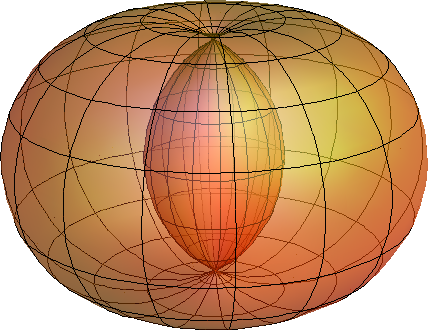}}
\caption{The standard tori represent the different types of Dupin cyclides.}
\label{fig:standard_tori}
\end{figure}

The class of Dupin cyclides is preserved by M\"obius transformations.
In the M\"obius geometric setting only the three types ``ring'', ``spindle'', and ``horn'' remain,
as there is no distinguished point at infinity.
So far, this can be found in the compact survey \cite{Pinkall:1986:DupinscheZykliden} by Pinkall.

In Lie geometry,
which contains M\"obius geometry as a subgeometry,
the distinction between the different types of cyclides,
where the type is determined by the number of singular points,
becomes obsolete:
There are Lie sphere transformations which map points to proper spheres and vice versa,
in particular the number of singular points of a Dupin cyclide is not a Lie invariant.
The following elegant characterization of Dupin cyclides is a reformulation of
the Definition~\ref{def:dupin_cyclide_euclidean} in the projective model of Lie geometry.
It can be found implicitly already in~\cite{Klein:1926:GeometrieVorlesungen}
and in more detail in~\cite{Blaschke:1929:DG3},
while in~\cite{Pinkall:1985:DupinHypersurfaces} the analogous statement is proven for the general case of so-called Dupin hypersurfaces.

\begin{theorem}
A Dupin cyclide in $\RI^3$ corresponds to a polar decomposition of $\RP^{4,2}$
into projective planes $P^{(1)}$ and $P^{(2)}$ of signature $(++-)$.
\label{thm:dupin_cyclides_projective}
\end{theorem}

Indeed oriented spheres in Lie geometry are described as points on the \emph{Lie quadric}
\begin{equation*}
\mathcal{Q}^{4,2} = P(\LL^{4,2}) \subset \RP^{4,2}
\quad \text{where} \quad
\LL^{4,2} = \left\{ \widehat v \in \R^{4,2} \mid \langle \widehat v,\widehat v \rangle = 0\right\}
\end{equation*}
and spheres $s_1,s_2$ in oriented contact are modelled by polar points,
i.e. $s_i = [\what v_i]$ with $\langle \what v_1,\what v_2 \rangle = 0$
(cf. Appendix~\ref{sec:lie_geometry}).
Thus talking about polar planes $P^{(1)}$ and $P^{(2)}$ in Theorem~\ref{thm:dupin_cyclides_projective}
is another description of two families of spheres as in Definition~\ref{def:dupin_cyclide_euclidean}
This motivates

\begin{definition}[Cyclidic families of spheres]
Denote
\[ \cycplanes = \left\{ \text{Planes in } \RP^{4,2} \text{ with signature }(++-)\right\}.\]
A 1-parameter family of spheres
corresponding to a conic section $P \cap \mathcal{Q}^{4,2}$ with $P \in \cycplanes$
is called a \emph{cyclidic family of spheres}
(cf. Fig.~\ref{fig:cyclidic_planes}).
\label{def:cyclidic_planes}
\end{definition}

Two spheres $s_1,s_2$ in oriented contact span a \emph{contact element},
which consists of all spheres in oriented contact with $s_1$ and $s_2$ and 
is modelled as the projective line $\inc[s_1,s_2]$ contained in the Lie quadric $\mathcal Q^{4,2}$.
Such lines $L \subset \mathcal Q^{4,2}$ are called \emph{isotropic lines} and we define
\begin{equation*}
\L_0^{4,2} = \left\{ \text{Isotropic lines in } \RP^{4,2} \right\} \cong
\left\{ \text{Contact elements of } \what\R^3 \right\}.
\end{equation*}

\begin{figure}[htb]
\centering
 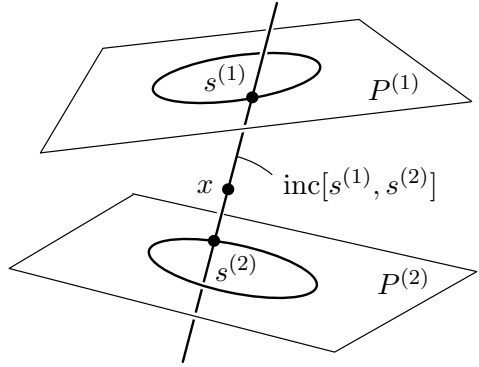 
\caption{Polar planes $P^{(1)},P^{(2)} \in \cycplanes$ describe a Dupin cyclide $\mathcal{C}$:
Each plane intersects the Lie quadric $\mathcal{Q}^{4,2}$ in a non-degenerate conic
$\mathcal{C}^{(i)} = \mathcal Q^{4,2} \cap P^{(i)}$,
and each conic corresponds to one of the 1-parameter families of spheres enveloping $\mathcal{C}$.
Each point $x$ of the cyclide is obtained as contact point of spheres
$s^{(1)} \in \mathcal{C}^{(1)}$
and
$s^{(2)} \in \mathcal{C}^{(2)}$,
the contact elements of $\mathcal{C}$ are given as
$\inc[s^{(1)},s^{(2)}] \subset \mathcal{Q}^{4,2}$.}
\label{fig:cyclidic_planes}
\end{figure}

To be precise, the Lie-geometric characterization of Dupin cyclides as elements of $\cycplanes$
coincides with Definition~\ref{def:dupin_cyclide_euclidean}
if one considers oriented spheres in oriented contact
and reads ``spheres'' as ``generalized oriented spheres''
(i.e. hyperspheres, hyperplanes, and points).
A circle then is also a Dupin cyclide:
One family of spheres consists of points (i.e. spheres of radius 0) constituting the circle,
the other family consists of all oriented spheres containing this circle
(including the two oriented planes).
This coincides with Lemma~\ref{lem:circle_planar_representatives},
which states that points in $\RI^3$ are concircular
if and only if their representatives in the Lie quadric are coplanar.
Note the implication: if three spheres in a cyclidic family are point spheres,
i.e. spheres of vanishing radius,
then the whole cyclide has to be a circle.
This corresponds to the fact that a Dupin cyclide in M\"obius geometry,
which is a 2-surface,
contains at most two singular points.
If one treats circles as Dupin cyclides,
the correspondence described in Theorem~\ref{thm:dupin_cyclides_projective} is indeed a bijection.

To simplify the presentation we exclude the circle case in the following considerations,
which allows us to see Dupin cyclides as surfaces in 3-space with at most two singular points.

\subsection{Curvature line parametrization of Dupin cyclides}
\label{subsec:dupin_cyclides_curvline}

\begin{proposition}
Curvature lines of Dupin cyclides are circles, each such circle being contained
in a unique enveloping sphere.  In particular, the enveloping spheres of a Dupin
cyclide are its principal curvature spheres.

Given a cyclide $\mathcal C$ and its enveloping spheres $\mathcal {C}^{(i)} = P^{(i)} \cap \mathcal{Q}^{4,2},\, i=1,2$,
points along the curvature line on $s^{(i)} \in \mathcal{C}^{(i)}$
are obtained as contact points of $s^{(i)}$ with all spheres of the complementary family $\mathcal{C}^{(j)}$ (cf. Fig.~\ref{fig:cyclidic_planes}).
Curvature lines from different families intersect in exactly one point.
\label{prop:curvature_spheres_touch_along_circles}
\end{proposition}

\begin{proof}
Let $s^{(i)} \in \mathcal{C}^{(i)}$ be a fixed sphere which is not a point sphere.
By Theorem~\ref{thm:dupin_cyclides_projective},
all contact elements of $\mathcal{C}$ containing $s^{(i)}$ are of the form
$\inc[s^{(i)},s^{(j)}],\, s^{(j)} \in \mathcal{C}^{(j)},\, i \ne j$.
They describe a circle $c^{(i)}$ on $s^{(i)}$ since $\mathcal{C}^{(j)}$ is contained in the plane $P^{(j)}$ 
(in this case $\inc[s^{(i)},P^{(j)}] \subset \pol[s^{(i)}]$ is 3-dimensional and one can use Proposition~\ref{prop:contact_element_cone_circle}).
Such circles are curvature lines,
since if a surface touches a sphere in a curve,
this curve is a curvature line
and the fixed sphere is the corresponding principal curvature sphere.
Two curvature circles $c^{(i)} \subset s^{(i)} \in \mathcal{C}^{(i)},\, i=1,2$,
intersect in the unique contact point of $s^{(1)}$ and $s^{(2)}$.
Since each non-singular point of the cyclide is contained in a unique contact element,
all curvature lines are circles.
\end{proof}

Proposition~\ref{prop:curvature_spheres_touch_along_circles} implies
that parametrizations of polar cyclidic families of spheres
induce a curvature line parametrization of the enveloped Dupin cyclide
(cf. Fig.~\ref{fig:cyclidic_planes}).

\begin{lemma}[Parametrization of conics]
Let $s_1,s_2,s_3$, $s_i = [\what s_i]$, be three points on a non-degenerate conic $\mathcal{C}$ in a projective plane $P$ and let
\[ Q = \begin{pmatrix} 0 & a_{12} & a_{13} \\ a_{12} & 0 & a_{23} \\ a_{13} & a_{23} & 0 \end{pmatrix} \]
be the matrix representation of the corresponding quadratic form $q$ in the
basis $(\what s_1, \what s_2, \what s_3)$, so that
$s_1 \leftrightarrow [1 , 0 , 0 ]^T,\, s_2 \leftrightarrow [0,1,0]^T, s_3
\leftrightarrow [0,0,1]^T$
and $a_{ij}= q(\what s_i , \what s_j)$.
Then $s = [\what s] : \RI \to P$ defined by
\begin{equation}
\what s (t) =
 \left( \begin{array}{l} a_{23} \cdot (2t^2 - 3t + 1) \\ a_{13} \cdot (t - t^2) \\ a_{12} \cdot (2t^2 - t) \end{array} \right)
\label{eq:parametrization_of_conics}
\end{equation}
is a parametrization of $\mathcal{C}$ for which
$s(0) = s_1, \, s(\frac 1 2) = s_2, \, s(1) = s_3$.
\label{lem:parametrization_of_conics}
\end{lemma}

This is a standard result obtained by projecting a conic $\mathcal C$ from a point $p \in \mathcal C$ to a line,
see \cite{BrannanEsplenGray:2002:Geometry} for example.
A direct computation shows
$\what s^T Q \what s = 0$ for all $t \in \RI$,
as well as the proposed normalization.

Applied to cyclidic families of spheres this yields the following Proposition
(according to later application it is convenient to use upper indices $i$):

\begin{proposition}[Parametrization of a cyclidic family of spheres]
Let $\mathcal{C}^{(i)} = P^{(i)} \cap \mathcal{Q}^{4,2}, P^{(i)} \in \cycplanes$,
be a cyclidic family of spheres which contains spheres $s^{(i)}_1,
s^{(i)}_2, s^{(i)}_3$, $s^{(i)}_j = [\what s^{(i)}_j]$.
A parametrization $s^{(i)} = [\what s^{(i)}]: \RI \to \mathcal{C}^{(i)}$
which maps $0 \mapsto s^{(i)}_1, \, \frac 1 2 \mapsto s^{(i)}_2, \, 1 \mapsto s^{(i)}_3$,
is given by quadratic polynomials
\begin{equation}
\what s^{(i)} (t) =
(\what s^{(i)}_1,\what s^{(i)}_2,\what s^{(i)}_3)
\cdot
 \left( \begin{array}{l} a^{(i)}_{23} \cdot (2t^2 - 3t + 1) \\ a^{(i)}_{13}
 \cdot (t - t^2) \\ a^{(i)}_{12} \cdot (2t^2 - t) \end{array} \right),
 \label{eq:parametrization_of_cyclidic_family}
\end{equation}
with
$a^{(i)}_{jk} := \langle \what s^{(i)}_j ,\what  s^{(i)}_k \rangle$.
\label{prop:parametrization_of_cyclidic_family}
\end{proposition}

\begin{proof}
The parametrization \eqref{eq:parametrization_of_cyclidic_family} follows immediately from \eqref{eq:parametrization_of_conics}
since the plane $P^{(i)}$ is spanned by the points $s^{(i)}_1,s^{(i)}_2,s^{(i)}_3 \in \mathcal{C}^{(i)}$.
(There are no three collinear points on $\mathcal{C}^{(i)}$ due to the signature $(++-)$ of $P^{(i)}$,
which means that $\mathcal C^{(i)}$ is a non-degenerate conic).
\end{proof}

\begin{theorem}[Curvature line parametrization of Dupin cyclides]\hfill
Let $s^{(i)} = [\what s^{(i)}] : \RI \to \mathcal{C}^{(i)}, i=1,2$,
be parametrizations \eqref{eq:parametrization_of_cyclidic_family}
of the cyclidic families of spheres enveloping a Dupin cyclide $\mathcal{C} \subset \RI^3$.
A (Lie geometric) curvature line parametrization of $\mathcal{C}$ in terms of contact elements is given by
\begin{equation}
L : \RI \times \RI \to \L_0^{4,2},\quad
(t_1,t_2) \mapsto \emph{inc}[s^{(1)}(t_1),s^{(2)}(t_2)].
\label{eq:parametrization_of_cyclides_lie}
\end{equation}
The corresponding Euclidean curvature line parametrization is given by the
$\e_1$-, $\e_2$-, and $\e_3$-components of the induced parametrization
\begin{equation}
\label{eq:parametrization_of_cyclides}
\what x : \RI \times \RI \to \LL^{4,2},
\quad
(t_1,t_2) \mapsto \frac{\xi(t_1,t_2)}{-2 \langle \xi(t_1,t_2) , \e_\infty \rangle},
\end{equation}
where
\begin{equation*}
\xi(t_1,t_2) =
\langle \what s^{(2)}(t_2), \e_r \rangle \ \what s^{(1)}(t_1) 
- \langle \what s^{(1)}(t_1), \e_r \rangle \ \what s^{(2)}(t_2).
\end{equation*}
%$\what x : \RI \times \RI \to \LL^{4,2}$ according to
%\begin{equation}
%\label{eq:parametrization_of_cyclides}
%\what x(t_1,t_2) =
%\frac{1}{\alpha(t_1,t_2)}
%\left(\langle \what s^{(2)}(t_2), \e_r \rangle \ \what s^{(1)}(t_1) 
%- \langle \what s^{(1)}(t_1), \e_r \rangle \ \what s^{(2)}(t_2)\right),
%\end{equation}
%where
%\begin{equation*}
%\alpha(t_1,t_2) =
%2 \left(\langle \what s^{(1)}(t_1), \e_r \rangle \ \langle \what s^{(2)}(t_2),
%\e_\infty \rangle - \langle \what s^{(2)}(t_2), \e_r \rangle \ \langle \what
%s^{(1)}(t_1), \e_\infty \rangle\right).
%\end{equation*}
For $\langle \xi(t_1,t_2) , \e_\infty \rangle \ne 0$, the right hand side of
\eqref{eq:parametrization_of_cyclides} are normalized homogeneous coordinates
\eqref{eq:representative_point} of the unique finite contact point of
\eqref{eq:parametrization_of_cyclides_lie}, which is the corresponding point of
the cyclide. The value $\langle \xi(t_1,t_2) , \e_\infty \rangle = 0$
corresponds to $x = \infty$.
\label{thm:parametrization_of_cyclides}
\end{theorem}

\begin{remark}
Note that $\what x$ can be constant along parameter lines (corresponding to
singular points of a cyclide), but $L$ cannot.
\end{remark}

\begin{proof}
Definition~\ref{def:dupin_cyclide_euclidean} translates to
\eqref{eq:parametrization_of_cyclides_lie}, and because of
Proposition~\ref{prop:curvature_spheres_touch_along_circles} this is a curvature
line parametrization of $\mathcal{C}$ in Lie geometry.  The parametrization
\eqref{eq:parametrization_of_cyclides} is obtained along the lines of
Lemma~\ref{lem:hom_coords_of_contact_points}, where division 
by the $\e_0$-component of $\xi$ yields the desired normalization.
\end{proof}

\subsection{Cyclidic patches}
\label{subsec:cyclidic_patches}

\begin{definition}[Cyclidic patch]
A \emph{cyclidic patch} is an oriented surface patch,
obtained by restricting a curvature line parametrization
of a Dupin cyclide
to a closed rectangle $I_1 \times I_2$.
We call a patch \emph{non-singular} if it does not contain singular points.
\label{def:cyclidic_patch}
\end{definition}

Geometrically, a cyclidic patch is a piece cut out of a Dupin cyclide along curvature lines,
i.e. along circular arcs, as in Fig.~\ref{fig:patch_cut}.
\begin{figure}[htb]
\centering
    \includegraphics[scale=.18]{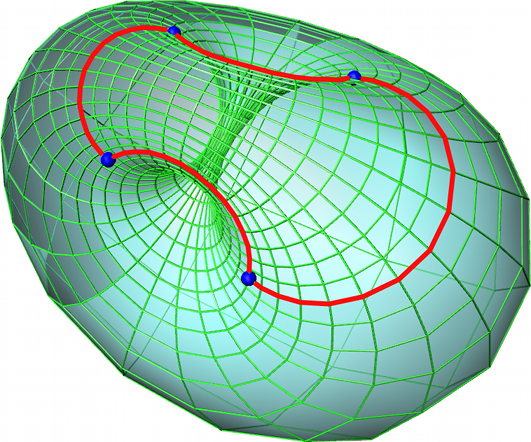}\hspace{2cm}
    \includegraphics[scale=.18]{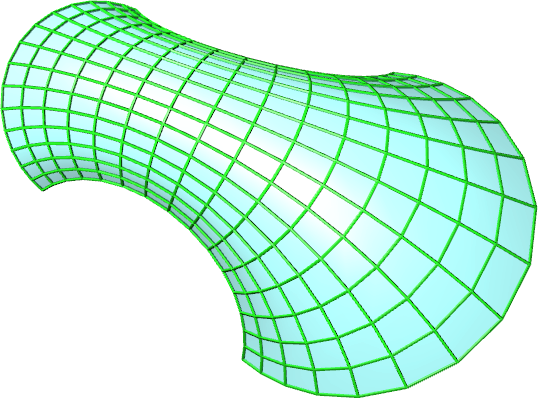}
\caption{A cyclidic patch.}
\label{fig:patch_cut}
\end{figure}
If the vertices of a patch are non-singular points,
the boundary consists of four circular arcs which intersect orthogonally at vertices.
Moreover, a patch is non-singular if and only if its boundary curves intersect only at vertices.

Now we introduce a notation, which will be convenient for the description of
cyclidic nets in Section~\ref{sec:cyclidic_nets}: If $f:[0,1] \times [0,1] \to
\R^3$ is a (parametrized) cyclidic patch, we denote its vertices by
\begin{equation*}
x = f(0,0),\
x_1 = f(1,0),\
x_{12} = f(1,1),\
x_2 = f(0,1).
\end{equation*}
The notation $(x,x_1,x_{12},x_2)$ refers to the vertices in this order.  Further
we write $\delta_i x = x_i-x$ and $\delta_i x_j = x_{ij}-x_j$ respectively.
Finally, we denote by $\widehat{x,x_i}$ the boundary arc of a cyclidic patch
connecting the vertices $x$ and $x_i$.

\begin{definition}[Vertex frames and boundary spheres of a cyclidic patch]
Let $f$ be a cyclidic patch with vertices $(x,x_1,x_{12},x_2)$.
The \emph{vertex frame} of $f$ at $x$
is the orthonormal 3-frame $B = (t^{(1)},t^{(2)},n)$,
where
$t^{(i)}$ is the tangent vector of $\widehat{x,x_i}$ directed from $x$ to $x_i$
and $n$ is the normal of the supporting cyclide.
The \emph{boundary spheres} of $f$ are the principal curvature spheres supporting the boundary arcs.
\end{definition}

The vertex frames always adapt to the boundary arcs of a patch.
For example the frame $B_1 = (t_1^{(1)},t_1^{(2)},n_1)$ at $x_1$ contains the unit vectors $t_1^{(1)}$ and $t_1^{(2)}$,
where $t_1^{(1)}$ is the tangent vector of $\widehat{x,x_1}$ directed from $x_1$ to $x$,
and $t_1^{(2)}$ is the tangent vector of $\widehat{x_1,x_{12}}$ directed from $x_1$ to $x_{12}$
(cf. Fig.~\ref{fig:patch_geometry}).

The perpendicular bisecting hyperplane of the segment $[x,x_i]$ we denote
\begin{equation}
    h^{(i)} = \left\{ p \in \R^N \mid
        \langle p,\delta_i x \rangle
    = \langle x + \tfrac12 \delta_i x,\delta_i x \rangle
    = \tfrac12 (\|x_i\|^2 - \|x\|^2)
    \right\}.
\label{eq:bisecting_hyperplane}
\end{equation}
By $H^{(i)}$ we denote the reflection in $h^{(i)}$ which interchanges $x$ and $x_i$.
A vector $v$ at $x$ is accordingly mapped to $v_i$ at $x_i$ due to
\begin{equation}
v_i = H^{(i)} (v) = v - 2 \, \frac{\langle \delta_i x,v \rangle}{\langle \delta_i x,\delta_i x \rangle} \, \delta_i x.
\label{eq:bisecting_hyperplane_map}
\end{equation}
The corresponding reflections in the hyperplanes $h_i^{(j)}$ are denoted by $H_i^{(j)}$.
(cf. Fig.~\ref{fig:reflection_lemma}).
These mappings extend naturally to tuples of vectors and we denote the image of a frame $B$ by $H^{(i)}(B)$.

\begin{lemma}
Let $x,x_1,x_{12},x_2$ be four distinct points on a circle.
Then the reflections \eqref{eq:bisecting_hyperplane_map} in hyperplanes \eqref{eq:bisecting_hyperplane} satisfy
\begin{equation}
H_1^{(2)} \circ H^{(1)} = H_2^{(1)} \circ H^{(2)}.
\label{eq:reflection_lemma_2-2}
\end{equation}
\label{lem:reflection_lemma}
\end{lemma}

\begin{proof}
Since the considered points are concircular they lie in a 2-plane $\Pi$.
All hyperplanes intersect in the affine orthogonal complement of $\Pi$ through the center of the circle
(cf. Fig.~\ref{fig:reflection_lemma}).
\end{proof}

\begin{figure}[htb]
\centering
 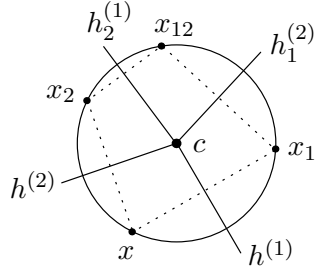 
\caption{Perpendicular bisecting hyperplanes between concircular points intersect
in the axis of the circle.}
\label{fig:reflection_lemma}
\end{figure}

\begin{proposition}[Geometry of cyclidic patches]
For a cyclidic patch with vertices $(x,x_1,x_{12},x_2)$ one has:
\begin{enumerate}[i)]
\item The vertices of the patch lie on a circle.
\item Vertex frames at neighboring vertices are related
by reflection in the bisecting plane \eqref{eq:bisecting_hyperplane},
possibly composed with the direction reversion of the non-corresponding tangent vector.
The change of orientation depends on whether the vertex quadrilateral
$(x,x_1,x_{12},x_2)$ is embedded or not.
For example frames $B$ and $B_i$ at vertices $x$ and $x_i$ are related as follows:
\begin{itemize}
\item If the points $x$ and $x_i$ on the circle are not separated by $x_j$ and $x_{12}$, then
\begin{equation}
B_i = H^{(i)} (B).
\label{eq:vertex_frames_reflection}
\end{equation}
\item If $x$ and $x_i$ are separated by $x_j$ and $x_{12}$, then
\begin{equation}
B_i = (H^{(i)} \circ F^{(j)}) (B) = (F^{(j)} \circ H^{(i)}) (B),\quad i \ne j
\label{eq:vertex_frames_reflection_nonembedded}
\end{equation}
where $F^{(j)}$ denotes the orientation change of $t^{(j)}$ in $(t^{(1)},t^{(2)},n)$.
\end{itemize}
In particular, lines $x+\R n$ and $x_i+\R n_i$
intersect in the center of the boundary sphere containing the boundary arc $\widehat{x,x_i}$
 (cf. Fig.~\ref{fig:patch_geometry}).
\end{enumerate}
\label{prop:geometric_properties_of_cyclidic_patches}
\end{proposition}

\begin{figure}[htb]
\centering
    \includegraphics[scale=.1]{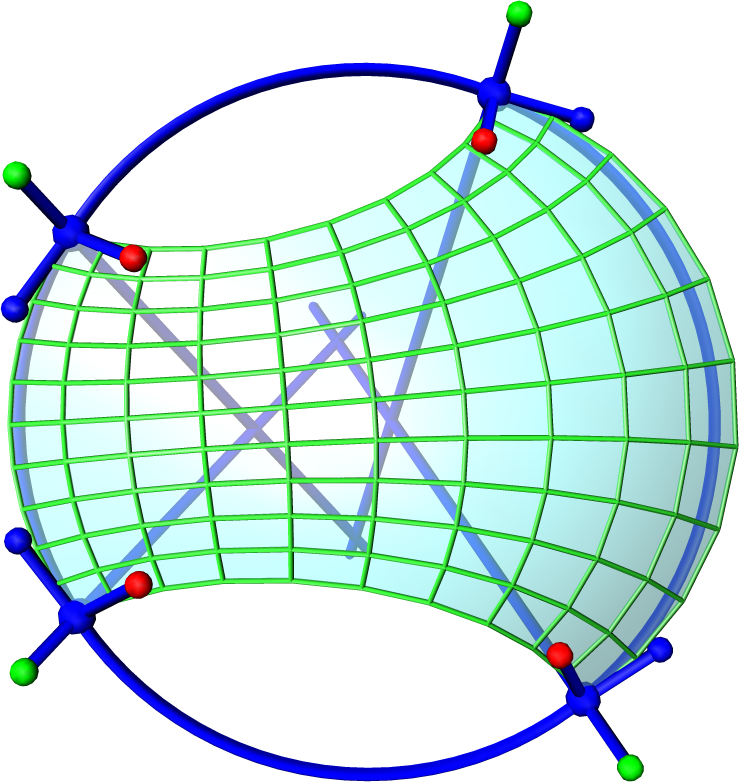}\quad\quad\quad
    \includegraphics[scale=.11]{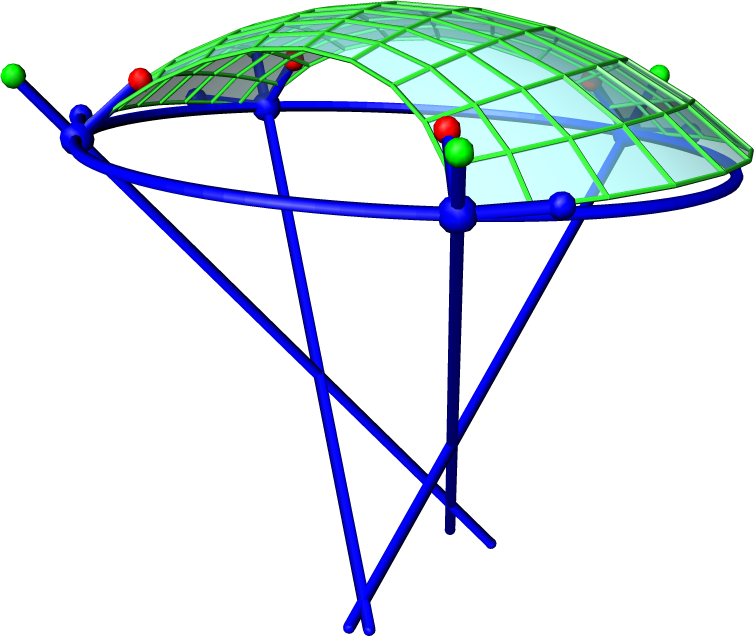}\quad\quad\quad
    \includegraphics[scale=.1]{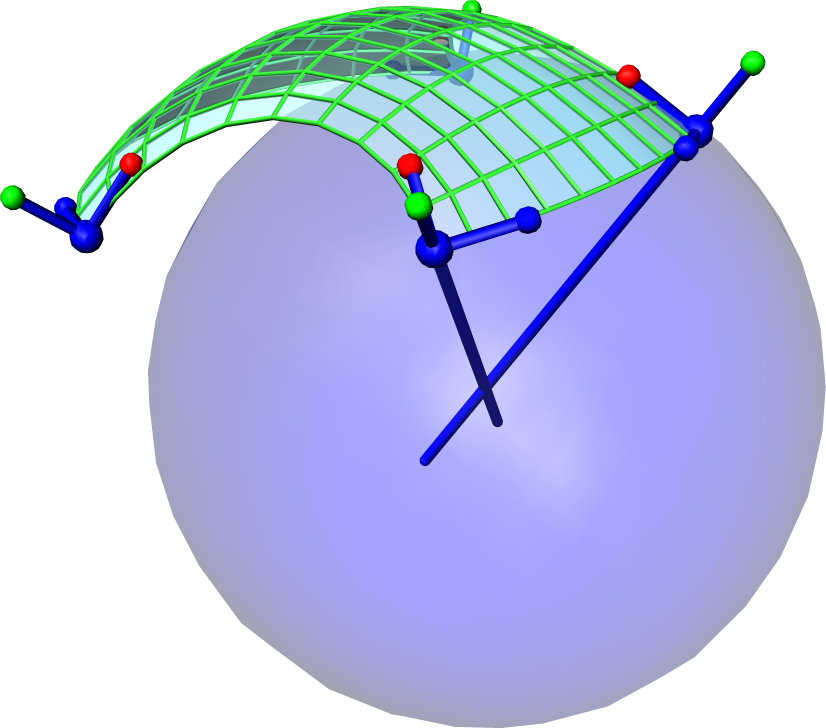}
\caption{Geometry of a cyclidic patch.}
\label{fig:patch_geometry}
\end{figure}

\begin{proof}
\textit{i)}
Consider the cyclidic families of spheres $\mathcal C^{(1)},\mathcal C^{(2)}$ enveloping the cyclide
supporting a cyclidic patch $f$.
Denote the boundary spheres of $f$ by $s^{(1)},s_2^{(1)} \in \mathcal{C}^{(1)}$ and $s^{(2)},s_1^{(2)} \in \mathcal{C}^{(2)}$.
The contact elements at vertices of $f$ are given by
$L=\inc[s^{(2)},s^{(1)}],L_1=\inc[s^{(1)},s_1^{(2)}],L_{12}=\inc[s_1^{(2)},s_2^{(1)}]$ and $L_2=\inc[s_2^{(1)},s^{(2)}]$
(cf. Fig.~\ref{fig:patch_contact_elements_generic}).
The space $V = \inc[L,L_1,L_{12},L_2] \subset \RP^{4,2}$ is at most 3-dimensional,
as $V = \inc[s^{(1)},s_1^{(2)},s_2^{(1)},s^{(2)}]$.
Since $V$ contains proper spheres, the intersection $V \cap \P(\e_r^\perp)$ is at most 2-dimensional.
By Lemma~\ref{lem:circle_planar_representatives} the vertices $(x,x_1,x_{12},x_2)$ are concircular.
\begin{figure}[htb]
\centering
 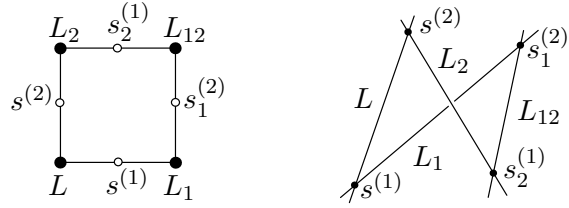 
\caption{Combinatorics and geometry of contact elements
(as isotropic lines in $\mathcal Q^{4,2}$)
at vertices of a cyclidic patch.}
\label{fig:patch_contact_elements_generic}
\end{figure}

\textit{ii)}
The sphere $s^{(i)}$ is symmetric with respect to $h^{(i)}$ which implies $H^{(i)}(n)=n_i$.
Moreover $H^{(i)}(t^{(i)})=t^{(i)}_i$,
since each circular arc containing $x$ and $x_i$ is symmetric with respect to $h^{(i)}$.
As $H^{(i)}$ is an isometry,
orthogonality is preserved and $H^{(i)}(t^{(j)}) = \pm t^{(j)}_i$ for $i \ne j$.
Thus $B$ and $B_i$ are necessarily related by \eqref{eq:vertex_frames_reflection}
or \eqref{eq:vertex_frames_reflection_nonembedded} and it remains to distinguish those cases:
Since we know that a cyclidic patch contains either 0,1, or 2 singular points,
%Figs.~\ref{fig:patch_singular_points} and \ref{fig:patch_singular_points_nonembedded} show all possible cases
Fig.~\ref{fig:patch_singular_points} shows all possible cases
(up to M\"obius transformations).
It turns out that $H^{(i)}(t^{(j)}) = - t^{(j)}_i$,
i.e. \eqref{eq:vertex_frames_reflection_nonembedded} holds,
if and only if $B$ and $B_i$ are of same orientation.
This happens if and only if $x,x_i$ are separated by $x_j,x_{12}$ on the circle through vertices.
\end{proof}

\paragraph{Spherical patches.}
\label{par:spherical_patches}

A sphere $s$ may be seen as the limit $\varepsilon \to 0$ of a family
$\mathcal{C}_\varepsilon$ of Dupin cyclides.  In this case exactly one of the
two corresponding families of conics $\mathcal{C}_\varepsilon^{(i)} \subset
\mathcal{Q}^{4,2}$ converges to $s \in \mathcal{Q}^{4,2}$ (cf.
Fig.~\ref{fig:cyclidic_planes}).  We call the corresponding limits of cyclidic
patches \emph{spherical patches} and consider them as degenerated cyclidic
patches for a unified treatment.  Continuity implies that a spherical patch also
has circular boundary arcs, and that the crucial geometric properties of
cyclidic patches described in
Proposition~\ref{prop:geometric_properties_of_cyclidic_patches} hold as well
(cf. Figs.~\ref{fig:patch_geometry_spherical} and
\ref{fig:patch_contact_elements_spherical}).  Note that the supporting sphere is
determined by the boundary curves of a spherical patch, while in contrast the
boundary curves of a generic cyclidic patch are never contained in a single
sphere.

\pagebreak

\begin{figure}[htb]
\begin{center}
    \includegraphics[scale=.095]{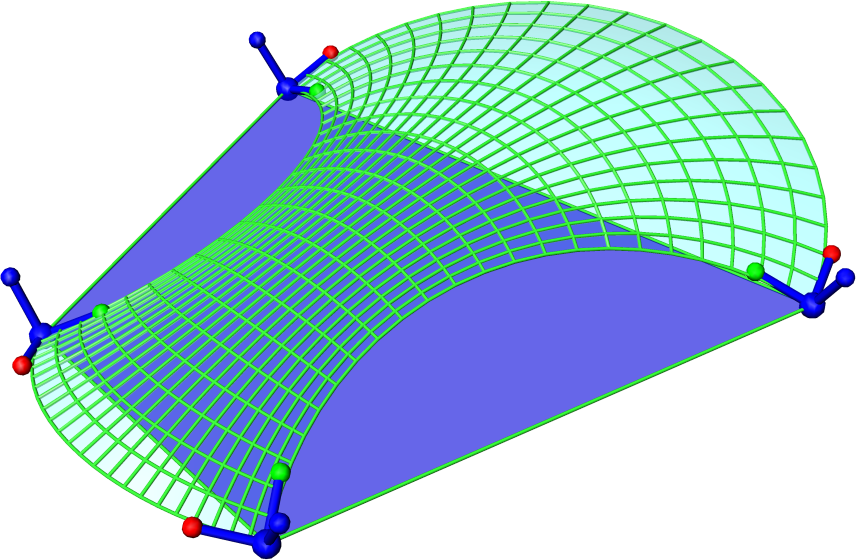}\quad\quad
    \includegraphics[scale=.095]{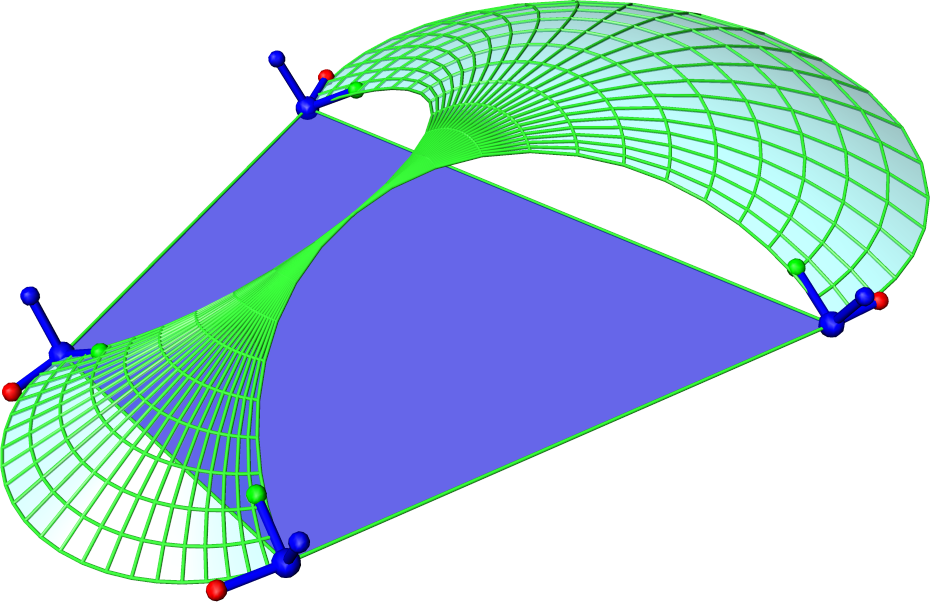}\quad\quad
    \includegraphics[scale=.095]{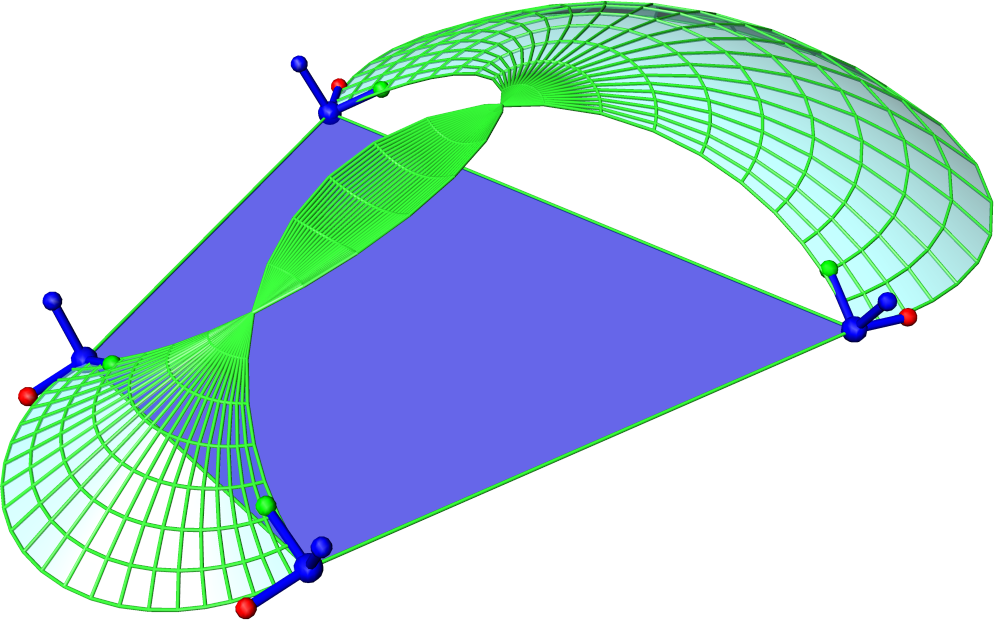}\\
		\medskip
    \includegraphics[scale=.14]{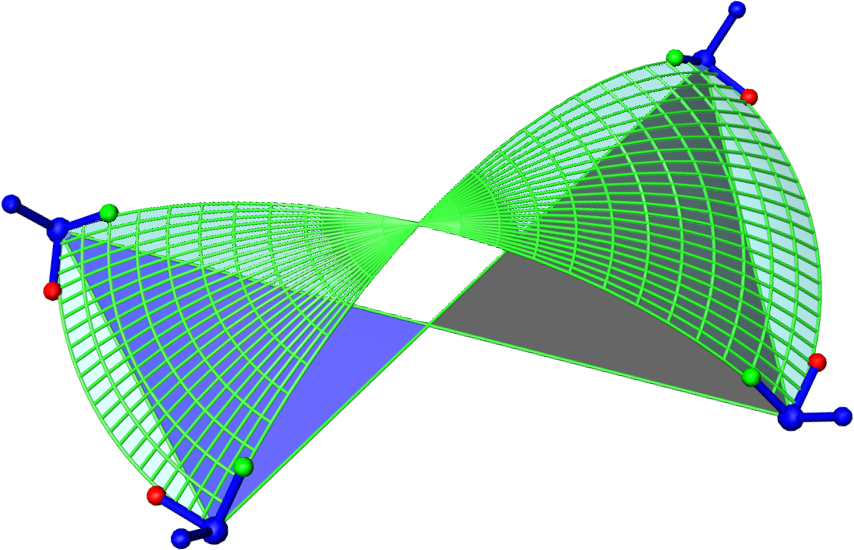}
\end{center}
\caption{Top: If the vertices $(x,x_1,x_{12},x_2)$ of a cyclidic patch build an embedded quadrilateral,
vertex frames at adjacent vertices have opposite orientation.
Bottom: Orientations of adjacent vertex frames of a cyclidic patch,
%for which the vertices $(x,x_1,x_{12},x_2)$ build a non-embedded quadrilateral.
for which the vertices build a non-embedded quadrilateral,
may coincide.
}
\label{fig:patch_singular_points}
\end{figure}

\begin{figure}[htb]
\begin{center}
    \includegraphics[scale=.1]{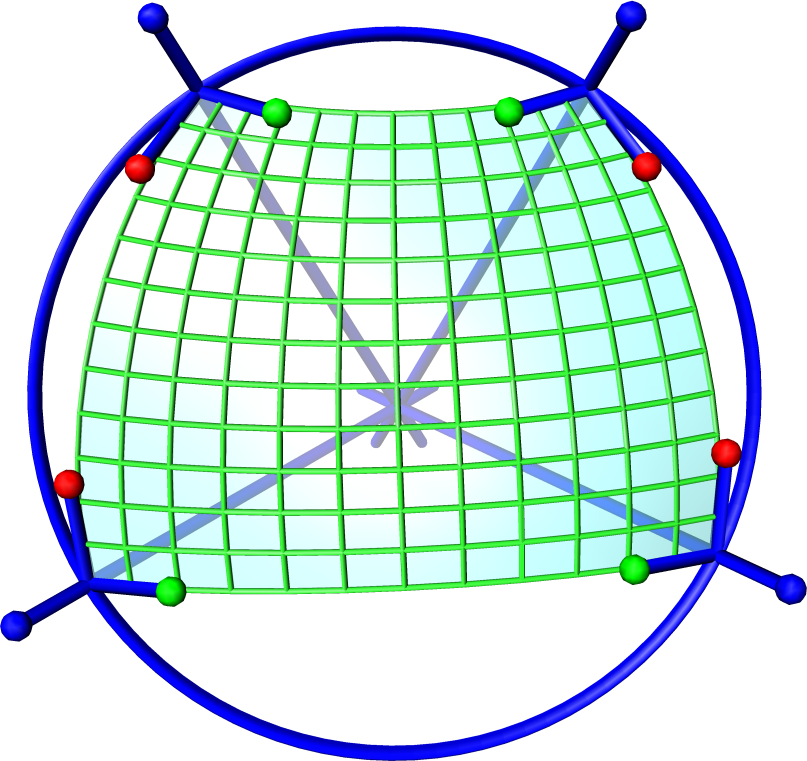}\quad\quad\quad
    \includegraphics[scale=.11]{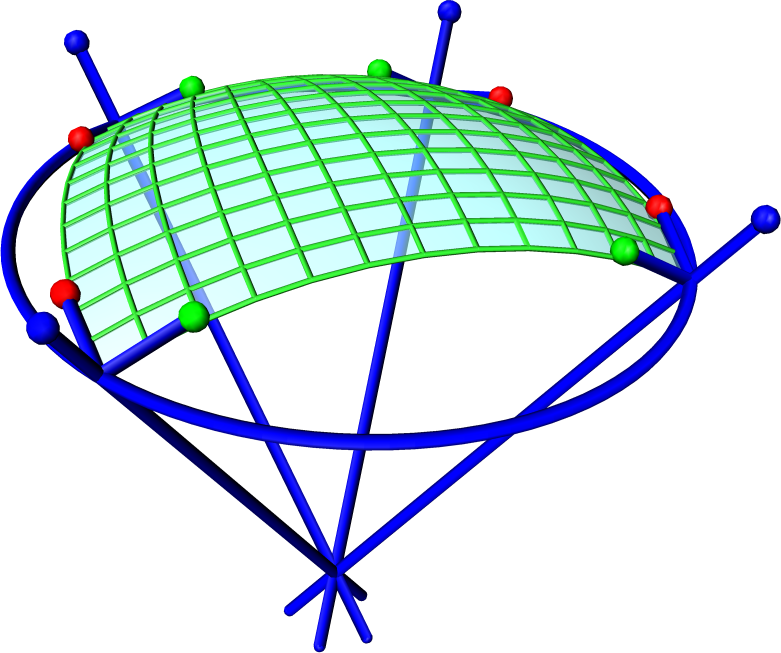}\quad\quad\quad
    \includegraphics[scale=.1]{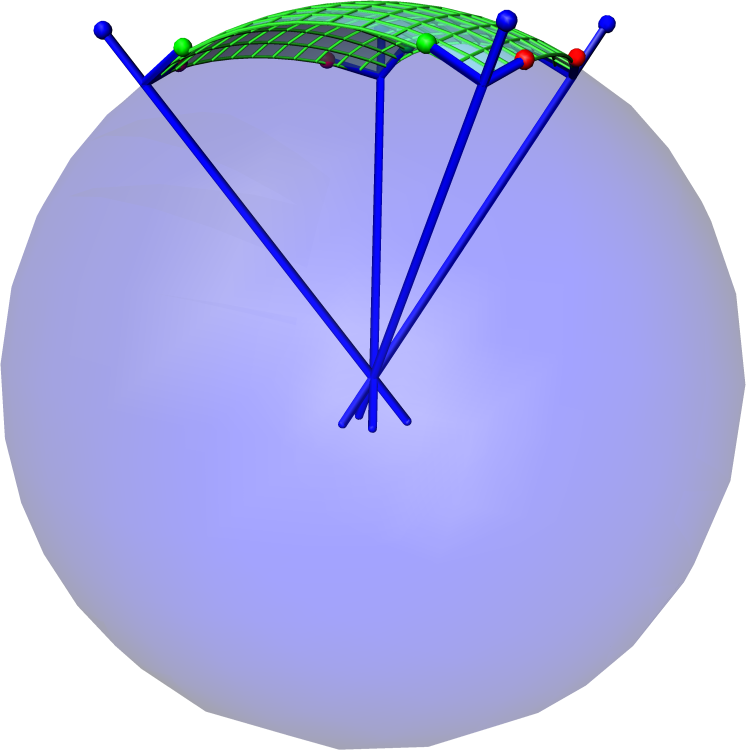}
\end{center}
\caption{Proposition~\ref{prop:geometric_properties_of_cyclidic_patches} also holds for spherical patches.}
\label{fig:patch_geometry_spherical}
\end{figure}

\begin{figure}[htb]
\begin{center}
 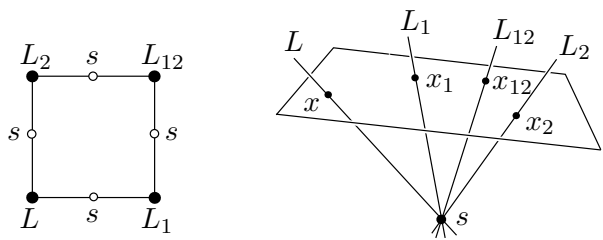 
\end{center}
\caption{Combinatorics and geometry of contact elements
(as isotropic lines in $\mathcal Q^{4,2}$) at vertices of a spherical patch.
The Vertices are concircular,
i.e. their representatives in the Lie quadric are coplanar,
and the contact elements span a 3-space.}
\label{fig:patch_contact_elements_spherical}
\end{figure}

\begin{proposition}
For a spherical patch there is a unique pair of orthogonal 1-parameter families of circles on the supporting sphere,
determined by the condition that the circular boundary arcs of the patch are contained in circles of those families.
\label{prop:ortho_circle_families}
\end{proposition}

\begin{proof}
Consider a spherical patch $f$ contained in a sphere,
and the four circles determined by the boundary arcs of $f$.
First suppose that the four circles are distinct and consider one pair of opposite circles.
Those two circles share either 0, 1, or 2 points.
These cases can be normalized by a M\"obius transformation to the cases of two concentric circles,
two parallel lines, or two intersecting lines.
In each case the two families of orthogonal circles are uniquely determined
(cf. Fig.~\ref{fig:ortho_pencils}).
If two boundary circles coincide,
the two others have to be distinct and one argues with the same normalization.
\begin{figure}[htb]
\begin{center}
 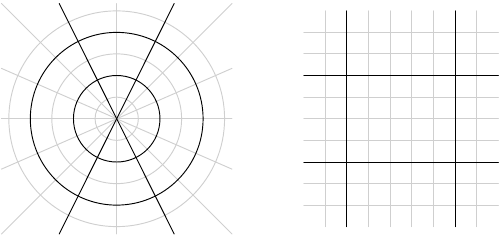 
\end{center}
\caption{There are two types of orthogonal pairs of 1-parameter families of circles.
The type of a pair is determined by two circles from one family.
}
\label{fig:ortho_pencils}
\end{figure}
\end{proof}

\pagebreak

A point $x$ of a spherical patch is a singular point
if all circular arcs of one family pass through $x$
(cf. Fig.~\ref{fig:ortho_pencils}).
Considering the two families of circles on the whole sphere,
one has exactly 2 or 1 points of this type.

Proposition~\ref{prop:ortho_circle_families} has the following consequence:

\begin{corollary}
For a spherical patch there exist orthogonal parametrizations with all coordinate lines being circular arcs.
The parameter lines are uniquely determined by the patch vertices and the boundary arcs.
\label{cor:ortho_param_of_spheres}
\end{corollary}

\subsection{Cyclidic patches determined by frames at vertices}

\begin{theorem}
Given four concircular points, there is a 3-parameter family of
cyclidic patches with these vertices.
Each such patch can be identified with an orthonormal 3-frame at one vertex.
\label{thm:patches_for_concircular_points}
\end{theorem}

\begin{proof}
Due to Proposition~\ref{prop:geometric_properties_of_cyclidic_patches}
it is enough to show that for concircular points $X=(x,x_1,x_{12},x_2)$
and any orthonormal 3-frame $B=(t^{(1)},t^{(2)},n)$
there is a unique cyclidic patch with $X$ as vertices and vertex frame $B$ at $x$.

Denote by $c_X$ the circle through $X$.
First obtain frames $B_1,B_{12},B_2$ at $x_1,x_{12},x_2$ from $B$ by successive
application of \eqref{eq:vertex_frames_reflection} and possibly \eqref{eq:vertex_frames_reflection_nonembedded},
depending on the ordering of the points $X$ on $c_X$.
The frames are well defined:
In the case that $X$ are vertices of an embedded quad,
this follows immediately from Lemma~\ref{lem:reflection_lemma}.
In the non-embedded case the additional orientation changes
involved in \eqref{eq:vertex_frames_reflection_nonembedded} cancel out
($B,B_i$ are related by \eqref{eq:vertex_frames_reflection_nonembedded}
if and only if $B_j,B_{ij}$ are as well).

The frames at vertices induce contact elements $\mathcal L = (L,L_1,L_{12},L_2)$,
where $L=(x,n)$, etc.
Due to \eqref{eq:vertex_frames_reflection}, \eqref{eq:vertex_frames_reflection_nonembedded},
contact elements $L,L_i$ share a sphere as depicted in Fig.~\ref{fig:principal_sphere}:
The intersection point $c^{(i)}$ is the center of a sphere $s^{(i)} = L \cap L_i$ with unoriented radius
$r^{(i)} = \|x-c^{(i)}\| = \|x_i - c^{(i)}\|$.
\begin{figure}[htb]
\centering
 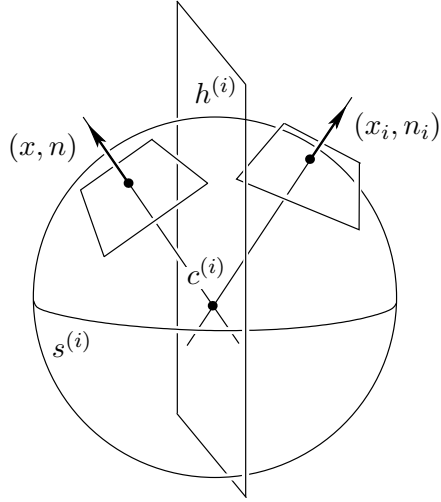 
\caption{Construction of proper boundary spheres of a cyclidic patch.}
\label{fig:principal_sphere}
\end{figure}
Note that $x,x_i$ and $B$ already determine a circular arc $\what{x,x_i} \subset \R^3$.
Since $x,x_i \in s^{(i)}$ and $t^{(i)}$ is tangential to $s^{(i)}$ at $x$,
we have $\what{x,x_i} \subset s^{(i)}$.

\noindent \textbf{The generic case}.
Generically,
i.e. if and only if $x + \R n$ and the axis of $c_X$ are skew,
$\mathcal{L}$ is as in Fig.~\ref{fig:patch_contact_elements_generic}.
Fix $i,j \in \left\{ 1,2 \right\}$ such that $i \ne j$
and denote $k^{(i)}$ the unique circle containing the circular arc $\what{x,x_i}$
determined by $x,x_i$ and $t^{(i)} \in B$.
First we show that there is a unique cyclide containing the contact elements $\mathcal{L}$,
for which $k^{(i)} \subset s^{(i)}$ is a curvature line:

Denote by $V \subset \pol[s^{(i)}]$ the 3-dimensional subspace
of $\RP^{4,2}$ corresponding to $k^{(i)}$
(cf. Proposition~\ref{prop:contact_element_cone_circle}).
According to Theorem~\ref{thm:dupin_cyclides_projective},
every cyclide can be described by a plane $P^{(j)} \in \cycplanes$
(determining $P^{(i)}$ via duality).
The corresponding cyclide touches $s^{(i)}$ in $k^{(i)}$ if and only if $V = \inc[P^{(j)},s^{(i)}]$,
see also Proposition~\ref{prop:curvature_spheres_touch_along_circles},
and this determines $P^{(j)}$ uniquely because of the following:
Due to the signature $(++-)$ of $P^{(i)}$,
the spheres $s^{(i)}$ and $s_j^{(i)}$ cannot be in oriented contact.
Therefore $s^{(i)} \in V$ implies $V \not\subset \pol[s_j^{(i)}]$.
Since $\dim V = 3,\, \dim P^{(j)} = 2$, and $P^{(j)} \subset \pol[s_j^{(i)}]$,
it follows $P^{(j)} = V \cap \pol[s_j^{(i)}]$.
The obtained cyclide $\mathcal{C}$ is independent of the choice of $i \in \left\{ 1,2 \right\}$,
since $k^{(j)}$ is also a curvature line of $\mathcal{C}$:
We know that $s^{(j)} \supset k^{(j)}$ is a curvature sphere of $\mathcal{C}$ and therefore
$\mathcal{C}$ touches $s^{(j)}$ in a unique circle $k$.
This circle $k$ is tangent to $t^{(j)}$ in $x$ due to the orthogonality of curvature lines.
Thus $x_j \in \mathcal{C} \cap s^{(j)}$ implies $x_j \in k$,
and we have $k = k^{(j)}$
as $k^{(j)}$ is the unique circle through $x$ and $x_j$ which is tangent to $t^{(j)}$ in $x$.
According to this, all four circles containing the circular arcs determined by
vertices $X$ and frames $(B,B_1,B_{12},B_2)$ are curvature line of $\mathcal{C}$.

Finally, the frame $B$ at $x$ determines the orientation of curvature lines of $\mathcal{C}$,
and hence a unique cyclidic patch for the given data (cf. Fig.~\ref{fig:torus_patch}).

\begin{figure}[htb]
\centering
 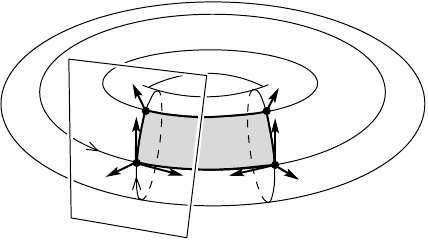 
\caption{ Vertices and vertex frames of a generic cyclidic patch determine the
supporting cyclide and the patch itself.  }
\label{fig:torus_patch}
\end{figure}

\noindent \textbf{The spherical case}. The contact elements
$\mathcal{L}$ are as in
Fig.~\ref{fig:patch_contact_elements_spherical} and share exactly one common
sphere if and only if $x + \R n$ intersects the axis of the circle $c_X$ or is
parallel to it. Clearly there is a 1-parameter family of such normals $n$, and
they are in bijection with all oriented spheres containing $c_X$. Denote the
corresponding 2-parameter family of frames by $\mathcal{B}_S$.  As there is a
3-parameter family $\mathcal{B}$ of orthonormal frames at $x$, for a frame $B
\in \mathcal{B}_S$ it is possible to choose a continuous variation of generic
frames $B_\varepsilon \in \mathcal{B} \setminus \mathcal B_S$ with
$B_\varepsilon \stackrel{\varepsilon \to 0}{\longto} B$.  According to our
consideration of spherical patches (cf. pp. \pageref{par:spherical_patches}ff)
this induces a continuous variation of cyclidic patches with fixed vertices and
gives a unique spherical patch in the limit.
\end{proof}

Next we describe the restriction of cyclidic patches to subpatches,
determined by the choice of points on the boundary arcs of the initial patch.

\begin{proposition}
For a cyclidic patch with notation as in Fig.~\ref{fig:sub_patch}, points
$\wtilde x_1$ and $\wtilde x_2$ on adjacent boundary arcs determine a unique
cyclidic subpatch.  If $\wtilde x_1$ and $\wtilde x_2$ are non-singular points,
the subpatch has four vertices $(x,\wtilde x_1,\wtilde x_{12},\wtilde x_2)$,
which are concircular, and its vertex frames satisfy
\eqref{eq:vertex_frames_reflection} resp.
\eqref{eq:vertex_frames_reflection_nonembedded} depending on the ordering of
vertices.
\begin{figure}[htb]
\centering
 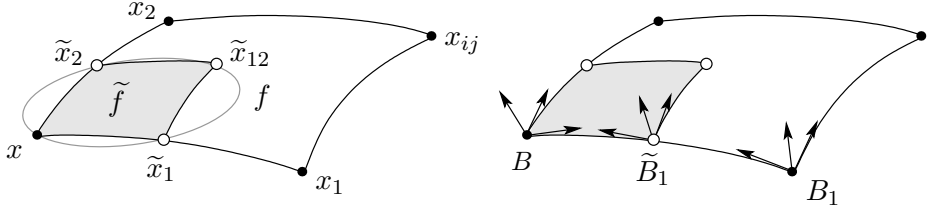 
\caption{Choosing points on the boundary determines a subpatch.}
\label{fig:sub_patch}
\end{figure}
\label{prop:sub_patches}
\end{proposition}

\enlargethispage{2\baselineskip}

\begin{proof}
\noindent \textbf{The generic case.}
Let $f : [0,1] \times [0,1] \to \RI^3$ be a curvature line parametrization of
the patch for which $x = f(0,0), \wtilde x_1 = f(u_0,0), \wtilde x_2 = f(0,v_0),
0 < u_0,v_0 < 1$.  This determines a unique subpatch $\wtilde f =
f\big|_{[0,u_0]\times[0,v_0]}$.  The fact that $x,\wtilde x_1, \wtilde x_2$ are
non-singular implies that $\wtilde x_{12} := f(u_0,v_0)$ is also non-singular
and, in particular, distinct from the initial three points.  Hence the claim
follows from Proposition~\ref{prop:geometric_properties_of_cyclidic_patches}.

\noindent \textbf{The spherical case.}
This is proven analogous to the spherical case of Theorem~\ref{thm:patches_for_concircular_points},
using a variation of frames $B_\varepsilon$ at $x$.
Note that this time one has to consider an additional variation of points $(\wtilde x_i)_\varepsilon$
inducing a variation $\wtilde f_\varepsilon$ of generic subpatches,
for which only the common vertex $x$ is fixed.
\end{proof}

\subsection{Curvature line parametrization of cyclidic patches}
\label{subsec:curvature_line_param_patches}

According to Theorem~\ref{thm:patches_for_concircular_points},
circular vertices $X=(x,x_1,x_{12},x_2)$ and an orthonormal 3-frame $B=(t^{(1)},t^{(2)},n)$ at $x$
determine a unique cyclidic patch.
In the generic case one obtains an Euclidean curvature line parametrization as follows:

\begin{enumerate}
\item[0)]
Identify the notation used in the context of Dupin cyclides,
with the notation for cyclidic patches:
Boundary spheres $s^{(i)},s^{(i)}_j, i \ne j$ of a patch correspond to
spheres $s_1^{(i)},s_3^{(i)}$ used in Proposition~\ref{prop:parametrization_of_cyclidic_family}.
The third sphere $s_2^{(i)}$ needed for a parametrization
\eqref{eq:parametrization_of_cyclidic_family}
of a cyclidic family of spheres is now denoted $\sigma^{(i)}$.
\item[1)]
Determine Lie coordinates \eqref{eq:representative_sphere} of the spheres
$(s^{(1)}, \sigma^{(1)}, s_2^{(1)})$ and
$(s^{(2)}, \sigma^{(2)}, s_1^{(2)})$
from the data $(X,B)$.
How this can be done is described in Proposition~\ref{prop:data_for_generic_patches}.
\item[2)]
Use the coordinates obtained in step 1) in \eqref{eq:parametrization_of_cyclidic_family}.
This gives parametrizations $\what s^{(1)},\what s^{(2)}$ of the cyclidic families of spheres
enveloping the supporting cyclide.
\item[3)]
Putting parametrizations $\what s^{(1)}$ and $\what s^{(2)}$ from step 2) in \eqref{eq:parametrization_of_cyclides}
yields a parametrization $\what x$ of the whole cyclide.
\item[4)]
A Euclidean curvature line parametrization of the patch is given by the $\e_1$-,
$\e_2$-, and $\e_3$- components of the restriction of $\what x$ to
$[0,1]\times[0,1]$.  In particular, $(\frac12,\frac12)$ is mapped to the contact
point of $\sigma^{(1)}$ and $\sigma^{(2)}$.
\end{enumerate}

All 3D pictures in this work except Fig.~\ref{fig:standard_tori} have been produced with
a program which implements the above described algorithm
(see also Section~\ref{subsec:implementation}).

\begin{proposition}
[Lie coordinates for curvature spheres of a cyclidic patch]
Let $f$ be a generic cyclidic patch with concircular vertices
$(x,x_1,x_{12},x_2)$ and vertex frame $B=(t^{(1)},t^{(2)},n)$ at $x$.
Further denote $\delta_i x = x_i - x$.

If $n \not \perp \delta_i x$,
Lie coordinates \eqref{eq:representative_sphere} of the boundary sphere $s^{(i)}$ of $f$ are
\begin{eqnarray}
\nonumber
\what s^{(i)} &=& x + \frac{\|\delta_i x\|^2}{2 \langle \delta_i x,n \rangle} n + \e_0 +
    \frac{\|x\|^2 \langle \delta_i x,n \rangle + \|\delta_i x \|^2 \langle x,n \rangle}{\langle \delta_i x,n \rangle} \e_\infty +\\
    &&  + \, \frac{\|\delta_i x\|^2}{2 \langle \delta_i x,n \rangle} \e_r.
\label{eq:boundary_sphere_from_frame}
\end{eqnarray}
Lie coordinates \eqref{eq:boundary_sphere_from_frame}
of the shifted boundary sphere $s^{(i)}_j$, $i\ne j$,
are obtained using the shifted formula:
$x \to x_j,\, \delta_i x \to \tau_j(\delta_i x) = x_{12} - x_j,\, n \to n_j$,
where the shifted normal $n_j$ at $x_j$ is given by
\begin{equation}
n_j = n - 2 \frac{\langle \delta_j x,n \rangle}{\langle \delta_j x,\delta_j x \rangle} \delta_j x.
\label{eq:shifted_normal}
\end{equation}
If $t^{(i)} \not \parallel \delta_i x$,
Lie coordinates \eqref{eq:representative_point} of the midpoint $y^{(i)}$ of the boundary arc $\widehat{x,x_i} \subset s^{(i)}$ are given by
\begin{equation}
\what y^{(i)} = x + \rho^{(i)} v^{(i)} + \e_0 + \|x + \rho^{(i)} v^{(i)}\|^2 \e_\infty,
\label{eq:arc_midpoint_from_frame}
\end{equation}
where
\begin{eqnarray}
v^{(i)} &=& \|\delta_i x\|t^{(i)} + \delta_i x
    \label{eq:midpoint_1} \\
\rho^{(i)} &=& \frac{\|\delta_i x\|^2}{2 \langle \delta_i x,v^{(i)} \rangle}
    = \frac{\|\delta_i x\|}{2 \langle \delta_i x,t^{(i)} \rangle + 2 \|\delta_i x\|}.
    \label{eq:midpoint_2}
\end{eqnarray}
For the enveloping sphere $\sigma^{(j)}$ of $f$ which touches $s^{(i)}$ in $y^{(i)}$ one obtains
\begin{equation}
\sigma^{(j)} = \left[ \langle \what y^{(i)}, \what s^{(i)}_j \rangle \, \what s^{(i)} -
\langle \what s^{(i)}, \what s^{(i)}_j \rangle \, \what y^{(i)} \right].
\label{eq:curvature_sphere_from_frame}
\end{equation}
Generically $\sigma^{(j)}$ is a proper sphere and its Lie coordinates \eqref{eq:representative_sphere}
are obtained by normalizing the $\e_0$-component of \eqref{eq:curvature_sphere_from_frame} to 1.
\label{prop:data_for_generic_patches}
\end{proposition}

\begin{proof}
Denote $h^{(i)}$ the bisecting plane \eqref{eq:bisecting_hyperplane}.
If $n \not \perp \delta_i x$,
the center $c^{(i)} = x + r^{(i)} n$ of $s^{(i)}$ is the intersection
$x + \R n \cap h^{(i)}$ as shown in Fig.~\ref{fig:principal_sphere}.
Solving $\langle x + r^{(i)}n,\delta_i x \rangle = \langle x + \frac{1}{2} \delta_i x, \delta_i x \rangle$ for $r^{(i)}$ gives
$r^{(i)} = \frac {\|\delta_i x\|^2}{2\langle \delta_i x,n \rangle}$.
Since $\|n\|=1$, the value $r^{(i)}$ is indeed the radius of $s^{(i)}$
and one easily checks that positive radius corresponds to the inward normal field.
Putting $c^{(i)}$ and $r^{(i)}$ in \eqref{eq:representative_sphere} gives \eqref{eq:boundary_sphere_from_frame}.
Here we used
\begin{equation*}
\|c^{(i)}\|^2 - (r^{(i)})^2
    = \|x\|^2 + 2 r^{(i)} \langle x,n \rangle
    = \frac{\|x\|^2 \langle \delta_i x,n \rangle + \|\delta_i x\|^2 \langle x,n \rangle}{\langle \delta_i x,n \rangle}.
\end{equation*}
The shifted normal $n_j$ at $x_j$ is obtained by reflection of $n$ in the bisecting plane
$h^{(j)}$ between $x$ and $x_j$,
i.e. \eqref{eq:shifted_normal} holds
(cf. \eqref{eq:bisecting_hyperplane_map}).

If $t^{(i)} \not \parallel \delta_i x$,
the midpoint $y^{(i)}$ of $\what{x,x_i}$ is the intersection of the angular bisector between
$\delta_i x$ and $t^{(i)}$ with $h^{(i)}$
(cf. Fig.~\ref{fig:boundary_arc_midpoint}-left):
\begin{figure}[htb]
\centering
 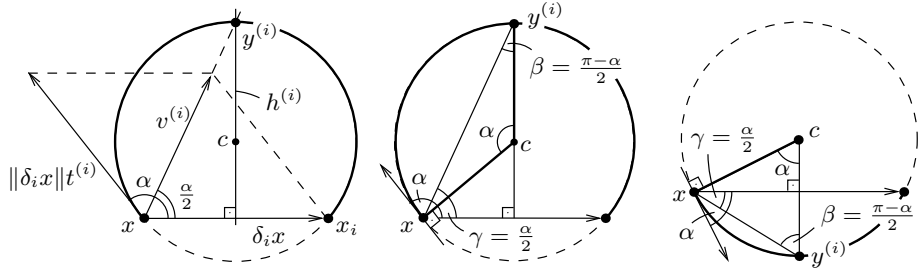 
\caption{Construction of the midpoint of a boundary arc.}
\label{fig:boundary_arc_midpoint}
\end{figure}
Let $\alpha \in (0,\pi)$ be the angle between $t^{(i)}$ and $\delta_i x$,
and denote $c$ the center of the circle determined by $x,x_i$ and its tangent $x+\R t^{(i)}$.
Then the angle $\angle (x,c,y^{(i)})$ also equals $\alpha$
(cf.
Fig.~\ref{fig:boundary_arc_midpoint}-right for $\alpha \in (0,\frac{\pi}{2})$
and Fig.~\ref{fig:boundary_arc_midpoint}-middle for $\alpha \in (\frac{\pi}{2},\pi)$).
Clearly $\|x-c\|=\|y^{(i)}-c\|$, which implies $\beta=\frac{\pi-\alpha}{2}$
and thus $\gamma = \frac{\pi}{2} - \beta = \frac{\alpha}{2}$.
So $y^{(i)} = x + \R v^{(i)} \cap h^{(i)}$
with $v^{(i)}=\|\delta_i x\|t^{(i)} + \delta_i x$ as $\|t^{(i)}\| = 1$.
This time solving $\langle x + \rho^{(i)} v^{(i)},\delta_i x \rangle = \langle x + \frac{1}{2} \delta_i x, \delta_i x \rangle$ for $\rho^{(i)}$ leads to
\[
y^{(i)} = x + \rho^{(i)} v^{(i)}, \quad \rho^{(i)} = \frac {\|\delta_i x\|^2}{2\langle \delta_i x,v^{(i)} \rangle}
    = \frac{\|\delta_i x\|}{2 \langle \delta_i x,t^{(i)} \rangle + 2 \|\delta_i x\|}.
\]
Thus for $t^{(i)} \not \parallel \delta_i x$
the Lie coordinates \eqref{eq:representative_point} of $y^{(i)}$ are exactly \eqref{eq:arc_midpoint_from_frame}.

The last formula \eqref{eq:curvature_sphere_from_frame} follows immediately from
$\sigma^{(j)} = \inc[s^{(i)},y^{(i)}] \cap \pol[s_j^{(i)}]$
by solving
$\langle \lambda \what s^{(i)} + \mu \what y^{(i)},\what s_j^{(i)} \rangle = 0$
(cf. Fig.~\ref{fig:common_polar_projecitve_point}).
The solution is unique (up to normalization of homogeneous coordinates)
as $s^{(i)}$ and $s_j^{(i)}$ are not in oriented contact for generic patches
(cf. p.~\pageref{unique_polar_sphere}).
\begin{figure}[htb]
\centering
 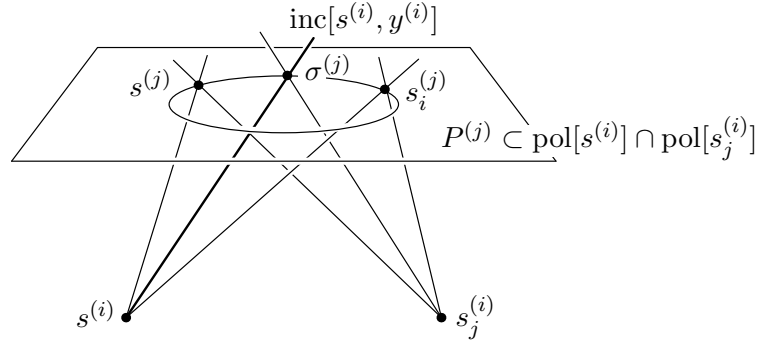 
\caption{There is a unique sphere $\sigma^{(j)}$ touching $s^{(i)}$ in $y^{(i)}$,
which is also in oriented contact with $s_j^{(i)}$.
This determines the plane $P^{(j)} \in \cycplanes$.
(Note that the shown configuration is 4-dimensional).}
\label{fig:common_polar_projecitve_point}
\end{figure}
\end{proof}

%\newpage

\begin{remarks}
\begin{enumerate}
\item Note that \eqref{eq:curvature_sphere_from_frame} can be applied using any point $y^{(i)}$ from the interior of the boundary arc $\what{x,x_i}$.
\item
If $n \perp \delta_i x$, the sphere $s^{(i)}$ becomes a plane with coordinates
$\what s^{(i)} = n + 2 \langle x,n \rangle \e_\infty + \e_r$.
This is obtained for example by
renormalizing \eqref{eq:boundary_sphere_from_frame} via multiplication with
$\frac{2 \langle \delta_i x,n \rangle}{\|\delta_i x\|^2}$
and then taking the limit $\langle \delta_i x,n \rangle \to 0$.
\item
Formula \eqref{eq:arc_midpoint_from_frame} also holds
if $\delta_i x = \|\delta_i x\|t^{(i)}$
and yields
$y^{(i)} = x + \frac 12 \delta_i x$
in this case.
This means the angle $\alpha$ in Fig.~\ref{fig:boundary_arc_midpoint}
equals zero, and $\what{x,x_i}$ is a straight line segment with midpoint $y^{(i)}$.

The case
$\delta_i x = - \|\delta_i x\|t^{(i)}$ is equivalent to $\eqref{eq:midpoint_1}$ becoming zero,
resp. to vanishing of the denominator in \eqref{eq:midpoint_2}.
Geometrically $\alpha = \pi$ and $\what{x,x_i}$ passes through $\infty$.
This time multiplying \eqref{eq:arc_midpoint_from_frame}
with $\langle \delta_i x,v^{(i)} \rangle$ and using $\|t^{(i)}\| = 1$,
one gets $y^{(i)} = \infty$ in the limit $v^{(i)} \to 0$.
\item
If $\sigma^{(j)}$ is a point sphere, its Lie coordinates
\eqref{eq:representative_point} are also obtained by normalizing the
$\e_0$-component of \eqref{eq:curvature_sphere_from_frame} to 1.  If
$\sigma^{(j)}$ turns out to be a plane, Lie coordinates
\eqref{eq:representative_plane} are obtained by normalizing
\eqref{eq:curvature_sphere_from_frame} such that the $\e_1$-, $\e_2$-, and
$\e_3$-components build a unit vector in $\R^3$.
\end{enumerate}
\end{remarks}

\section{Cyclidic nets}
\label{sec:cyclidic_nets}

The discretization of orthogonal nets as cyclidic nets is closely
related to the discretizations as circular, conical and principal
contact element nets (see the introduction). The idea behind
cyclidic nets is to construct piecewise smooth geometries by
gluing cyclidic patches in a differentiable way. For a unified
treatment we define cyclidic nets in a more abstract way in terms
of frames at vertices of a circular net, without referring to
cyclidic patches.

In $\R^3$, the realization of a 2D cyclidic net using cyclidic patches is a piecewise smooth $C^1$-surface,
to be considered as discretization of a curvature line parametrized surface
(cf. Fig.~\ref{fig:cyclidic_from_circular}).
A 3D cyclidic net in turn discretizes a triply orthogonal coordinate system,
by discretizing its coordinate surfaces as 2D cyclidic nets.
This corresponds to the classical Dupin theorem, which states that the restriction of a triply
orthogonal coordinate system to any of its coordinate surfaces (by fixing one coordinate)
yields a curvature line parametrization.
In the discrete case the corresponding 2D cyclidic nets indeed intersect orthogonally along their curvature lines,
although the definition of a 3D cyclidic net requires orthogonality only at vertices (cf. Fig.~\ref{fig:octrihedron}).

\subsection{Circular nets and definition of cyclidic nets}
\label{subsec:cyclidic_nets_definition}

We call a function $f$ on $\Z^m$ a \emph{discrete net} and write
\begin{equation*}
\tau_i f(z) = f(z + \e_i) = f_i(z),
\end{equation*}
where $\e_i$ is the unit vector of the $i$-th coordinate direction.
Mostly we will omit the variable $z$
and denote the functions values by $f,f_i,f_{ij}$, etc.

\begin{definition}[Circular net]
A map $x:\Z^m \to \RI^N$ is called a \emph{circular net},
if at each $z \in \Z^m$ and for all pairs $1 \le i < j \le m$
the points $(x,x_i,x_{ij},x_j)$ lie on a circle.
\label{def:circular_net}
\end{definition}

Circular nets are a discretization of smooth orthogonal nets and objects of M\"obius geometry
(see \cite{BobenkoSuris:2008:DDGBook} for details).
For $N=3$, the case $m = 2$ is a discretization of
curvature line parametrized surfaces in $\RI^3$,
where the case $m = 3$ discretizes triply orthogonal coordinate systems.
A circular net $x:\Z^3 \to \R^N$ is determined by its values along
three coordinate planes intersecting in one point.
This is due to the following classical result
(see for example \cite{BobenkoSuris:2008:DDGBook} for the proof).

\begin{theorem}[Miquel theorem]
Let $x,x_i$, and $x_{ij}$ $(1 \le i < j \le 3)$ be seven points in $\RI^N$,
such that each of the three quadruples $(x,x_i,x_{ij},x_j)$ lies on a circle $c_{ij}$.
Define three new circles $\tau_k c_{ij}$ as those passing through the point triples
$(x_k,x_{ik},x_{jk})$.
Generically these new circles intersect in one point:
$x_{123} = \tau_1 c_{23} \cap \tau_2 c_{13} \cap \tau_3 c_{12}$.
\label{thm:miquel}
\end{theorem}

We call an elementary hexahedron of a circular net a \emph{spherical cube},
since all its vertices are contained in a 2-sphere which we call \emph{Miquel sphere}
(cf. Fig.~\ref{fig:spherical_cube}).

\begin{figure}[htb]
\centering
\includegraphics[scale=.18]{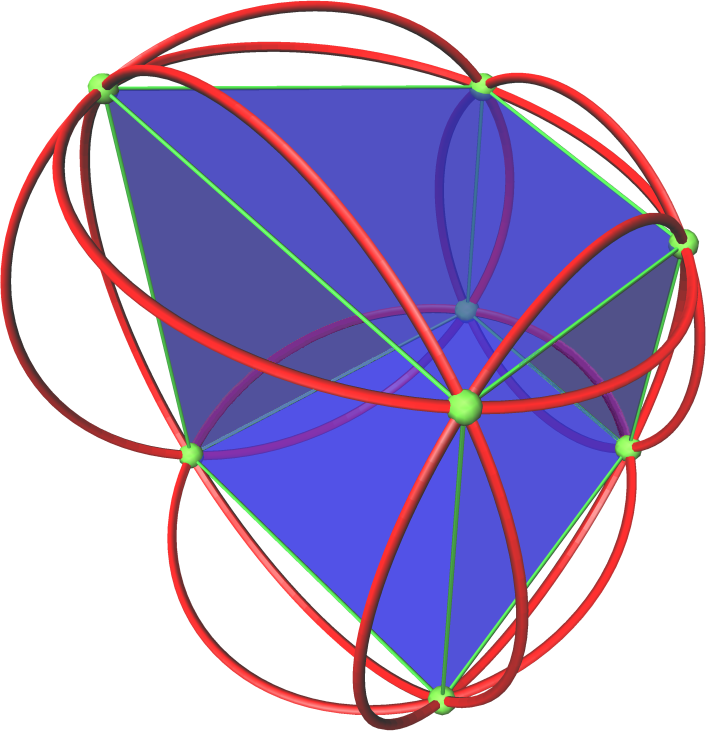}
\caption{
A spherical cube:
The eight vertices of a cube are contained in a 2-sphere
if and only if each face is a circular quadrilateral.
Any seven vertices determine the eighth one uniquely.}
\label{fig:spherical_cube}
\end{figure}

Cyclidic nets are defined as circular nets with an additional structure:

\begin{definition}[Cyclidic net]
A map
\[
(x,B) : \Z^m \to \RI^N \times \left\{ \text{orthonormal $m$-frames } B = (t^{(1)},\dots,t^{(m)}) \right\}, \quad m \le N
\]
is called an \emph{$m$D cyclidic net},
if
\begin{enumerate}[i)]
\item
The net $x$ is a circular net for which all quadrilaterals $(x,x_i,x_{ij},x_j)$ are embedded, $1\le i < j \le m$.
\item
The frames $B,B_i$ at neighboring vertices $x,x_i$ are related by
\begin{equation}
B_i = (H^{(i)} \circ F^{(i)}) (B) = (F^{(i)} \circ H^{(i)}) (B), \quad 1 \le i \le m,
\label{eq:frame_evolution}
\end{equation}
where $H^{(i)}$ is the reflection $\eqref{eq:bisecting_hyperplane_map}$ in the bisecting hyperplane \eqref{eq:bisecting_hyperplane} and
$F^{(i)}$ maps
$(t^{(1)},\dots,t^{(i)},\dots,t^{(m)})$ to $(t^{(1)},\dots,-t^{(i)},\dots,t^{(m)})$.
\end{enumerate}
\label{def:cyclidic_net}
\end{definition}

\begin{figure}[htb]
\centering
 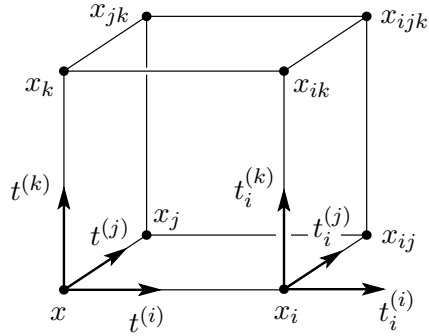 
\caption{Notation for cyclidic nets.}
\label{fig:cycnet_notation}
\end{figure}

\paragraph{Extension of circular nets to cyclidic nets.}
\label{par:circular_to_cyclidic}

The idea is to choose an initial frame $B_0$ at $x_0$,
and then evolving this frame according to \eqref{eq:frame_evolution}.
This is always possible due to

\begin{proposition}
Given a circular net $x:\Z^m \to \R^N$ with all elementary quadrilaterals embedded
and an orthonormal $m$-frame $B_0$ in $x_0 = x(z_0)$,
there is a unique cyclidic net respecting this data.
\label{prop:frame_evolution_consistency}
\end{proposition}

\begin{proof}
As an immediate consequence of Lemma~\ref{lem:reflection_lemma} the evolution \eqref{eq:frame_evolution} is consistent,
i.e. yields a unique frame at every vertex.
\end{proof}

Let $k = \sum_{i=1}^{m} N-i$ for $m \le N$.  As there is a $k$-parameter family
of orthonormal $m$-frames $B_0$ at $x_0 \in \R^N$, an $m$D circular net $x:\Z^m
\to \R^N$ can be extended to a $k$-parameter family of cyclidic nets.

\subsection{Geometry of cyclidic nets}
\label{subsec:cyclidic_nets_geometry}

Given the vertices of a cyclidic net,
the frames determine a unique cyclidic patch
for every elementary quadrilateral (cf. Theorem~\ref{thm:patches_for_concircular_points}).
The evolution of frames guarantees that patches associated to edge-adjacent quads
share a boundary arc,
and supposed all involved patches are non-singular the join of patches belonging to the same coordinate directions is $C^1$.
Nevertheless, since the frames are arbitrary, the involved patches may be singular.

Our main interest is the discretization of smooth geometries, thus in the
following we assume that all considered patches are non-singular.  In
particular, this implies that all vertex quadrilaterals have to be embedded (cf.
Fig.~\ref{fig:patch_singular_points}), which is already incorporated in the
Definition~\ref{def:cyclidic_net} of cyclidic nets.

\subsection*{Geometry of 2D cyclidic nets in $\R^3$}

\paragraph{2D cyclidic nets as $C^1$-surfaces.}
\label{par:2d_cycnets_as_c1}

Consider a 2D cyclidic net $(x,B)$ in $\R^3$.
Identify the orthonormal 2-frames $B = (t^{(1)},t^{(2)})$ with the
3-frames $(t^{(1)},t^{(2)},n)$, where $n = t^{(1)} \times t^{(2)}$.
According to Theorem~\ref{thm:patches_for_concircular_points},
the data $(x,x_1,x_{12},x_2)$ and $(t^{(1)},t^{(2)},n)$ at $x$ determines a unique cyclidic patch.
There are three more circular quads containing the vertex $x$.
Since all quads are embedded,
for each corresponding patch the vertex frames are related by \eqref{eq:vertex_frames_reflection}.
Comparison with \eqref{eq:frame_evolution} yields
that frames fit as shown in Fig.~\ref{fig:2d_cycnets_are_c1}.
This implies that patches sharing two vertices
also share the corresponding boundary sphere and boundary arc
(see Fig.~\ref{fig:principal_sphere} and recall that a circle is determined by two points and one tangent).
Since the boundary sphere is tangent to both patches along the common boundary arc
one obtains a piecewise smooth surface which is $C^1$ across the joints.
This consideration also shows that one should associate circular arcs and boundary spheres to the edges of $\Z^2$.

\begin{figure}[htb]
\centering
\includegraphics[scale=.25]{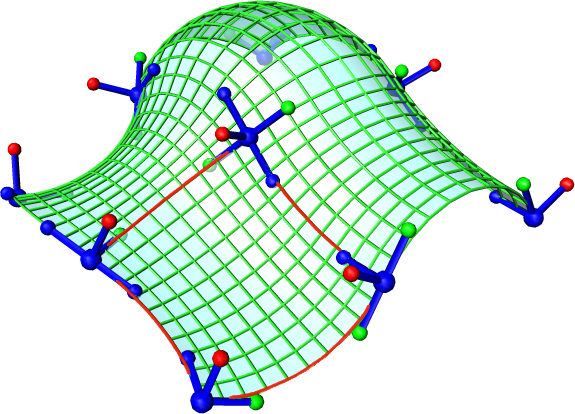}
\hspace{1.8cm}
\includegraphics[scale=.27]{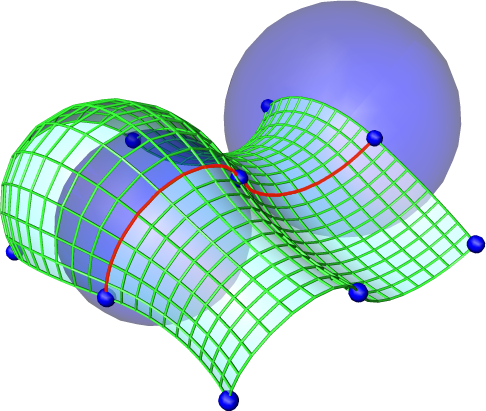}
\caption{A generic 2D cyclidic net is a $C^1$-surface.}
\label{fig:2d_cycnets_are_c1}
\end{figure}

\paragraph{Discrete curvature lines.}
\label{par:discrete_curv_lines}

For a 2D cyclidic net the boundary arcs of its patches can be
treated as discrete curvature lines. These arcs constitute a net
of $C^1$-curves with $\Z^2$ combinatorics on the surface, which
intersect orthogonally at vertices. The curves are determined by
the data $(x,B)$, without referring to cyclidic patches.

\paragraph{Relation to principal contact element nets.}

For a 2D cyclidic net $(x,B)$ with frames $B = (t^{(1)},t^{(2)})$ the
normals at patch vertices are given by $n = t^{(1)} \times t^{(2)}$.
Those normals induce a net $(x,n)$ of contact elements.
As explained in the proof of Theorem~\ref{thm:patches_for_concircular_points},
contact elements $(x,n),(x_i,n_i)$ share an oriented sphere (cf. Fig.~\ref{fig:principal_sphere}).
This is the defining property for a net of contact elements $(x,n)$
to constitute a so-called \emph{principal} contact element net.
Principal contact element nets are a discretization of curvature line parametrized surfaces in Lie geometry,
and they unify the previous discretizations
as circular nets in M\"obius geometry and as conical nets in Laguerre geometry.
For details and further references see \cite{BobenkoSuris:2008:DDGBook}.

Note that principal contact element nets discretize literally the
Lie geometric characterization of curvature lines, i.e.
infinitesimally close contact elements along a curvature line
share the corresponding principal curvature sphere.

\paragraph{Parallel 2D cyclidic nets.}

As any $C^1$-surface, a 2D cyclidic net $(x,B)$ possesses a family
of parallel surfaces. In particular, the offset of a 2D cyclidic
net is again a 2D cyclidic net. The corresponding discrete data is
obtained by offsetting the circular net $x \mapsto x + \varepsilon
n$ and keeping the frames $B=(t^{(1)},t^{(2)},n)$ (Indeed one
obtains another cyclidic net, since the offset surface possesses
the same reflections \eqref{eq:bisecting_hyperplane_map}).

\paragraph{Ribaucour transformation of 2D cyclidic nets.}

According to the ``multidimensional consistency principle'' for the discretization of smooth geometries,
the transformations of cyclidic nets are defined by the same geometric relations as the nets itself.
For a detailed explanation see \cite[Chap. 3]{BobenkoSuris:2008:DDGBook},
where Ribaucour transformations for principal contact element nets are introduced in the same spirit.

\begin{definition}[Ribaucour transformation of cyclidic nets]
Cyclidic nets
\[(x,B),(x^+,B^+) : \Z^2 \to \RI^3 \times \{\text{orthonormal 3-frames } B = (t^{(1)},t^{(2)},n)\}\]
are called \emph{Ribaucour transforms} of each other, if
for all $z \in \Z^2$ the corresponding frames $B$ and $B^+$ are related by
\begin{equation*}
B^+ = (H^{(+)} \circ F^{(+)}) (B) = (F^{(+)} \circ H^{(+)}) (B).
\end{equation*}
Here $H^{(+)}$ denotes the reflection in the perpendicular bisecting plane between $x$ and $x^+$,
and $F^{(+)}$ denotes the orientation change $(t^{(1)},t^{(2)},n) \mapsto (t^{(1)},t^{(2)},-n)$.
\label{def:ribaucour_trafo_cyclidic}
\end{definition}

The existence and standard properties including Bianchi's permutability property
can be derived exactly as for principal contact element nets.
Definition~\ref{def:ribaucour_trafo_cyclidic} means that two 2D cyclidic nets in the
relation of Ribaucour transformation build a 2-layer 3D cyclidic net.
For each pair of corresponding vertices $x$ and $x^+$ there is a sphere tangent to the
cyclidic net $(x,B)$ at $x$ and tangent to $(x^+,B^+)$ at $x^+$
(cf. Fig.~\ref{fig:principal_sphere}).

\subsection*{Geometry of 3D cyclidic nets in $\R^3$}

3D cyclidic nets in $\R^3$ are a discretization of triply orthogonal coordinate systems.
By Dupin's theorem,
all 2-dimensional coordinate surfaces of such a coordinate system are curvature line parametrized surfaces,
i.e. coordinate surfaces from different families intersect each other orthogonally along curvature lines.
The discretization as cyclidic nets preserves this property
in the sense that the 2D layers of a 3D cyclidic net
(seen as $C^1$-surfaces)
intersect orthogonally along the common previously discussed discrete curvature lines
(cf. p.~\pageref{par:discrete_curv_lines}).
Moreover, we will show that the natural discrete families of coordinate surfaces
of a 3D cyclidic net can be extended to piecewise smooth families of 2D cyclidic nets.

\paragraph{Elementary hexahedra of 3D cyclidic nets.}

We call an elementary hexahedron of a 3D cyclidic net a \emph{cyclidic cube}.
By definition of cyclidic nets,
the vertices of a cyclidic cube form a spherical cube,
i.e. they are contained in a 2-sphere.

Now consider a 3D cyclidic net $(x,B)$ in $\R^3$.
According to Theorem~\ref{thm:patches_for_concircular_points},
for each circular quadrilateral $(x,x_j,x_{jk},x_k)$
one has a cyclidic patch $f^{(i)}$ with these vertices and the normal $t^{(i)}$ in $x$.
Therefore we obtain six cyclidic patches as faces of a cyclidic cube:
The faces $f^{(i)},f^{(j)},f^{(k)}$ are all determined by the same frame $B$ at $x$,
while the shifted faces are determined by the shifted frames,
i.e. $f^{(i)}_i$ is determined by $B_i$ at $x_i$ etc.
(cf. Fig.~\ref{fig:patches_for_elem_quads} and Fig.~\ref{fig:cyclidic_cube}).

\begin{figure}[htb]
\centering
 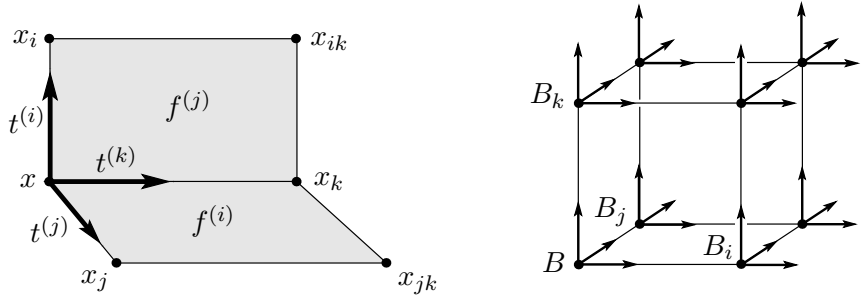 
\caption{
    Left:   Notation for cyclidic patches associated to elementary quadrilaterals of a cyclidic net.
    Right: Combinatorics of frames of a cyclidic net, obtained by the evolution \eqref{eq:frame_evolution}.
    }
\label{fig:patches_for_elem_quads}
\end{figure}

\begin{proposition}[]
Adjacent faces of a cyclidic cube
intersect orthogonally along their common boundary arc.
\label{prop:patches_intersect_orthogonally_along_arcs}
\end{proposition}

\begin{figure}[htb]
\centering
\includegraphics[scale=.25]{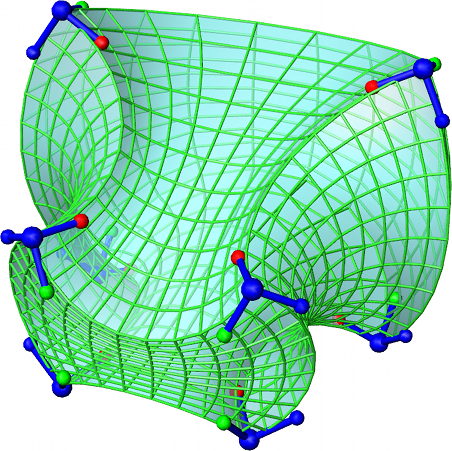} \hspace{1.7cm}
\includegraphics[scale=.25]{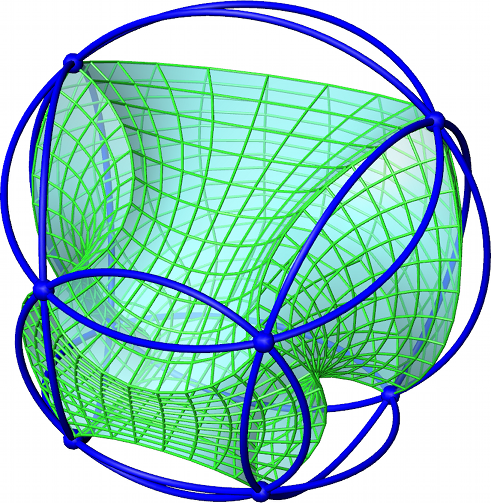}
\caption{A cyclidic cube.
The vertices are contained in a 2-sphere and adjacent faces intersect orthogonally.}
\label{fig:cyclidic_cube}
\end{figure}

\begin{proof}
First note that corresponding boundary arcs of adjacent patches coincide:
Since all quadrilaterals are embedded,
for each face the vertex frames are related by \eqref{eq:vertex_frames_reflection}.
On the other hand the evolution \eqref{eq:frame_evolution}
yields frames as in Fig.~\ref{fig:patches_for_elem_quads}.
For vertex frames of adjacent patches the tangents associated to the common edge coincide.
Since also the endpoints coincide, the corresponding boundary arcs coincide.

Orthogonal intersection in the common vertices
follows immediately from the orthogonality of the corresponding frames.
Knowing that corresponding boundary arcs of adjacent patches coincide,
orthogonal intersection along those arcs follows from Proposition~\ref{prop:sub_patches}
by considering vertex frames of subpatches with common endpoints
(for example, faces $f^{(i)}$ and $f^{(j)}$ intersect orthogonally
in each point $\wtilde x_k$ of the interior of $\what{x,x_k}$,
cf. Fig.~\ref{fig:orthogonally_intersecting_patches}).
\begin{figure}[htb]
\centering
 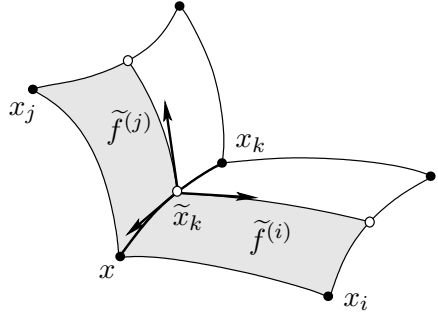 
\caption{Adjacent cyclidic patches associated to different coordinate directions in a 3D cyclidic net
intersect orthogonally along their common boundary arc.}
\label{fig:orthogonally_intersecting_patches}
\end{figure}
\end{proof}

\begin{proposition}[]
For a spherical cube there is a 3-parameter family of cyclidic cubes with same vertices.
These cyclidic cubes are in bijection with orthonormal 3-frames at one vertex.
\label{prop:miquel_to_cyclidic}
\end{proposition}

\begin{proof}
This is a special instance of the extension of circular nets to cyclidic nets
(cf. p.~\pageref{par:circular_to_cyclidic}).
\end{proof}

\begin{corollary}
Three cyclidic patches sharing one vertex and boundary arcs with combinatorics as in
Fig.~\ref{fig:three_faces_combinatorially},
can be extended uniquely to a cyclidic cube.
\label{cor:three_face_determine_cube}
\end{corollary}

\begin{figure}[htb]
\centering
 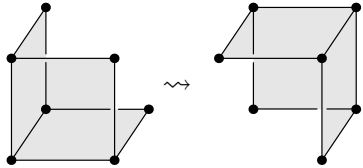 
\caption{Three faces determine a cyclidic cube.}
\label{fig:three_faces_combinatorially}
\end{figure}

\begin{proof}
Seven vertices of a spherical cube are already given and by
Theorem~\ref{thm:miquel} there is the unique eighth vertex.
Since the three patches determine the frame at the common vertex,
Proposition~\ref{prop:miquel_to_cyclidic} proves the claim.
\end{proof}

\paragraph{Discrete triply orthogonal coordinate system.}

The restriction of a 3D cyclidic net to a coordinate plane of $\Z^3$ is a 2D cyclidic net,
which we see as a $C^1$-surface (cf. p.~\pageref{par:2d_cycnets_as_c1}).
We obtain three discrete families of coordinate surfaces
\begin{equation*}
M^{(i)} = \left\{
    (x,B)\big|_{E_i(c)} \mid c \in \Z \right\},
\quad
E_i(c) = \left\{ (z_1,z_2,z_3) \in \Z^3 \mid z_i = c \right\}.
\end{equation*}
Three such coordinate surfaces $m^{(i)} = (x,B)\big|_{E_i(c_i)}, i=1,2,3$,
share the vertex $x(c_1,c_2,c_3)$.
By Proposition~\ref{prop:patches_intersect_orthogonally_along_arcs}
any two surfaces $m^{(i)},m^{(j)}$ intersect orthogonally in a common discrete curvature line
(cf. p.~\pageref{par:discrete_curv_lines}),
in particular all three surfaces intersect orthogonally in $x$
(cf. Fig.~\ref{fig:octrihedron}).
These intersection curves are considered as coordinate lines of the discrete triply orthogonal coordinate system.

\paragraph{Extension to piecewise smooth coordinate systems.}

A \emph{non-singular cyclidic cube} is a cyclidic cube
for which opposite faces are disjoint and adjacent faces intersect only along the common boundary arc.
If all elementary hexahedra of a 3D cyclidic net are non-singular,
then the discrete families $M^{(i)}$ of coordinate surfaces
can be extended to piecewise smooth families $\mathcal{M}^{(i)}$
such that the previously described orthogonality property still holds.

\begin{theorem}[]
Let $f^{(i)},f_i^{(i)}$ be a pair of opposite faces of a non-singular cyclidic cube.
For each point $\wtilde x_i \in \what{x,x_i}$
there is a unique cyclidic patch $\wtilde f^{(i)}$ such that
$\wtilde x_i$ is a vertex of $\wtilde f^{(i)}$ and
$\wtilde f^{(i)}$ intersects the four faces
$f^{(j)},f_j^{(j)},f^{(k)},f_k^{(k)}$ orthogonally along common curvature lines
(cf. Fig.~\ref{fig:patch_in_between}).
Moreover, the patch $\wtilde f^{(i)}$ is non-singular.
\label{thm:patch_in_between}
\end{theorem}

\begin{figure}[htb]
\centering
 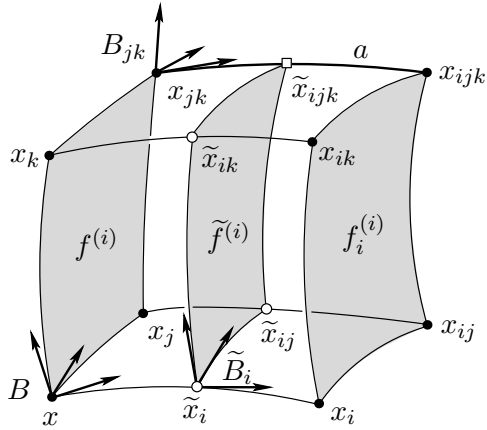 
\caption{Choosing $\wtilde x_i \in \what{x,x_i}$ determines the patch $\wtilde f^{(i)}$.
Boundary arcs of $\wtilde f^{(i)}$ are curvature lines of the (orthogonal) faces
$f^{(j)},f^{(k)},f_j^{(j)},$ and $f_k^{(k)}$.}
\label{fig:patch_in_between}
\end{figure}

\begin{proof}
Given a cyclidic cube and a point $\wtilde x_i \in \what{x,x_i}$,
following curvature lines of $f^{(k)}$ and $f^{(j)}$
one obtains points $\wtilde x_{ij} \in \what{x_j,x_{ij}}$ and
$\wtilde x_{ik} \in \what{x_k,x_{ik}}$ (cf. Fig.~\ref{fig:patch_in_between}).
Take the obtained points $\wtilde x_{ij},\wtilde x_{ik}$ and go along curvature lines
of the patches $f^{(j)}_j$, respectively $f^{(k)}_k$,
which gives $\bar x_{ijk},\bar x_{ikj} \in a := \what{x_{jk},x_{ijk}}$.
We will prove $\bar x_{ijk} = \bar x_{ikj} = \wtilde x_{ijk}$,
where $\wtilde x_{ijk}$ is the eighth vertex of the spherical cube with
remaining vertices $(x,\wtilde x_i,x_j,x_k,\wtilde x_{ij},\wtilde x_{ik},x_{jk})$.
According to Propositions~\ref{prop:sub_patches} and \ref{prop:miquel_to_cyclidic}
this implies the existence of a unique patch $\wtilde f^{(i)}$ as claimed.

Denote by $\wtilde s$ the Miquel sphere (see Theorem~\ref{thm:miquel}) determined
by the vertices $(x,\wtilde x_i,x_j,x_k)$, and by $s$ the Miquel sphere determined by
$(x,x_i,x_j,x_k)$.

Suppose $\bar x_{ijk} \ne \bar x_{ikj}$.
In this case $a \subset \wtilde s$,
as we have three distinct points $x_{jk}, \bar x_{ijk}, \bar x_{ikj} \in a \cap \wtilde s$,
which in turn gives $x_{ijk} \in \wtilde s$.
On the other hand $x_{ijk}$ is also contained in the Miquel sphere $s$.
We have the chain of implications
\begin{equation*}
\bar x_{ijk} \ne \bar x_{ikj}
\, \implies \, a \subset \wtilde s
\, \implies \, x_{ijk} \in \wtilde s
\, \implies \, s = \wtilde s
\, \implies \, a \subset s.
\end{equation*}
This proves the claim $\wtilde x_{ijk} = \bar x_{ijk} = \bar x_{jik}$
in the generic case $a \not \subset s$.

The general validity of $\bar x_{ijk} = \bar x_{ikj}$ can be shown using a continuity argument.
The crucial observation is
that $s$ is fixed by the vertices $(x,\dots,x_{ijk})$ of the initial cyclidic cube,
whereas the arc $a$ depends on the frame $B_{jk}=(t_{jk}^{(1)},t_{jk}^{(2)},t_{jk}^{(3)})$ at $x_{jk}$
(cf. Fig.~\ref{fig:patch_in_between}).
The vector $t_{jk}^{(i)}$ is tangent to $s$ if and only if $t^{(i)}$ is tangent to $s$,
due to the evolution \eqref{eq:frame_evolution}.
Now consider a frame $B$ at $x$ with $t^{(i)}$ tangent to $s$,
and as before choose $\wtilde x_i \in \what{x,x_i}$.
Obviously there exists an axis $l$ through $x$
such that rotating $B$ by an arbitrary small angle $\alpha \in (-\varepsilon,\varepsilon) \setminus \{0\}$ around $l$
has the effect that the image $t^{(i)}(\alpha)$ of $t^{(i)} = t^{(i)}(0)$ does not remain tangent to $s$.
Denote by $\what{x,x_i}(\alpha)$ the circular arc between $x$ and $x_i$ determined by the vector $t^{(i)}(\alpha)$
and choose $\wtilde x_i (\alpha) \in \what{x,x_i}(\alpha)$ s.t. $\wtilde x_i(0) = \wtilde x_i$.
If $\wtilde x_i(\alpha)$ is continuous in $\alpha$,
the dependent points $\bar x_{ijk}(\alpha),\bar x_{ikj}(\alpha)$ are continuous in $\alpha$ as well.
For all angles $\alpha \in (-\varepsilon,\varepsilon)\setminus\left\{ 0 \right\}$ we are in the generic case,
for which we know $\bar x_{ijk}(\alpha) = \bar x_{ikj}(\alpha)$.
Thus continuity implies $\bar x_{ijk}(\alpha) = \bar x_{ikj}(\alpha)$ in the limit $\alpha = 0$.

A cyclidic patch is non-singular if and only if opposite
%boundary curves do not intersect (see Figs.~\ref{fig:patch_singular_points}, \ref{fig:patch_singular_points_nonembedded}).
boundary curves do not intersect (see Fig.~\ref{fig:patch_singular_points}).
Since we started with a non-singular cyclidic cube, the patch $\wtilde f^{(i)}$ is non-singular.
\end{proof}

\begin{corollary}[]
Each pair $f^{(i)},f_i^{(i)}$ of opposite faces of a non-singular cyclidic cube
can be uniquely extended to a smooth family $\mathcal{F}^{(i)}$ of cyclidic patches
such that boundary curves of $\wtilde f^{(i)} \in \mathcal{F}^{(i)}$
are curvature lines of the initial faces
$f^{(j)},f_j^{(j)},f^{(k)},f_k^{(k)}$
(cf. Fig.~\ref{fig:patch_in_between}).
\label{cor:extension_to_smooth_families}
\end{corollary}

\begin{remark}
Corollary \ref{cor:extension_to_smooth_families} implies that opposite faces of
a cyclidic cube form a classical Ribaucour pair. In particular, for a Ribaucour
pair of 2D cyclidic nets corresponding patches form classical Ribaucour pairs.
\end{remark}

Moreover one has

\begin{corollary}[]
Through any interior point of a non-singular cyclidic cube there pass
exactly three (pairwise orthogonal) patches $\wtilde f^{(i)} \in \mathcal{F}^{(i)},\, i=1,2,3$
(cf. Corollary~\ref{cor:extension_to_smooth_families}).
\label{cor:orthogonal_coordinates}
\end{corollary}
\begin{proof}
For any interior point $\wtilde x_{123}$ of a non-singular cyclidic cube there exist
patches $\wtilde f^{(i)} \in \mathcal{F}^{(i)},\, i=1,2,3$, such that
$\wtilde x_{123} := \wtilde f^{(1)} \cap \wtilde f^{(2)} \cap \wtilde f^{(3)}$.
This follows from Corollary~\ref{cor:extension_to_smooth_families} using a continuity argument.
The patches are unique, since patches of a single family $\mathcal F^{(i)}$ are disjoint.
Indeed, suppose two patches from the same family, say $\mathcal F^{(1)}$, intersect in a point $x^*$.
Then $x^*$ is a singular point for each patch $f \in \mathcal F^{(2)} \cup \mathcal F^{(3)}$ with $x^* \in f$,
because patches from different families intersect along common curvature lines.
This contradicts Theorem~\ref{thm:patch_in_between} since we started with a non-singular cyclidic cube.
\end{proof}

Corollary~\ref{cor:extension_to_smooth_families} implies that
the discrete families $M^{(i)}$ of coordinate surfaces of a 3D cyclidic net in $\R^3$
can be extended to continuous families $\mathcal{M}^{(i)}$.
Now taking three 2D cyclidic nets $m^{(i)} \in M^{(i)}$ as in Fig.~\ref{fig:octrihedron}
determines three coordinate axes intersecting in $x_0 = m^{(1)} \cap m^{(2)} \cap m^{(3)}$,
and according to Theorem~\ref{thm:patch_in_between} parametrizations of the coordinate axes
induce parametrizations of the continuous families $\mathcal{M}^{(i)}$
(cf. Fig.~\ref{fig:cut_down_cube}).
By Corollary~\ref{cor:orthogonal_coordinates} one has orthogonal coordinates on an open subset of $\R^3$.

\begin{figure}[htb]
\centering
 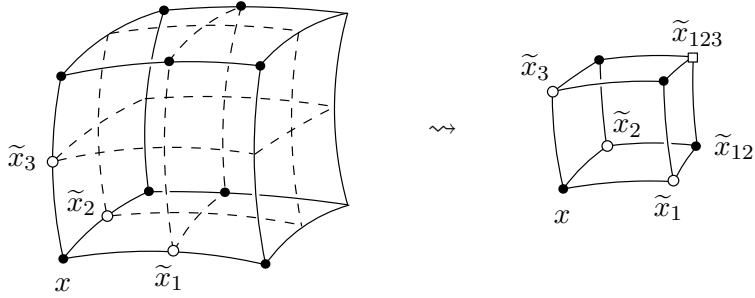 
\caption{Points on boundary arcs of a cyclidic cube as coordinates for points in the interior.}
\label{fig:cut_down_cube}
\end{figure}

\paragraph{Remark on the smoothness.}
The interior point $\wtilde x_{123}$ depends smoothly on the coordinates $\wtilde x_1,\wtilde x_2,\wtilde x_3$.
Indeed the points $\wtilde x_{12},\wtilde x_{23},\wtilde x_{13}$ depend smoothly on $\wtilde x_1,\wtilde x_2,\wtilde x_3$,
since the faces of a cyclidic cube are smooth surface patches.
But then also the point $\wtilde x_{123}$ as eighth vertex of the spherical cube with remaining vertices
$x,\wtilde x_i, \wtilde x_{ij}$ depends smoothly on $\wtilde x_1,\wtilde x_2,\wtilde x_3$.

\subsection{Convergence of cyclidic nets to curvature line pa\-ra\-me\-trized surfaces and orthogonal coordinate systems}
\label{subsec:convergence}

Convergence of circular nets to smooth curvature line parametrized surfaces and orthogonal coordinate systems was proven in
\cite{BobenkoMatthesSuris:2003:OrthogonalSystems}.
Starting with a smooth curvature line parametrization or orthogonal coordinate system $f:D \subset \R^m \to \R^N$,
a one-parameter family $f_\varepsilon$ of circular nets of corresponding dimension was constructed
which converges to $f$ with all derivatves in the limit $\varepsilon \to 0$.
Here the parameter $\varepsilon$ is the grid size parameter for the discretization grid approximating $D$.
Actually more was shown:
Not only convergence of orthogonal nets as point maps was analysed,
but also convergence of the associated discretization of frames at vertices was proven.
In the present paper we have given a geometric interpretation of these frames as vertex frames of cyclidic patches.
Therefore the results of
\cite{BobenkoMatthesSuris:2003:OrthogonalSystems}
imply the following

\begin{theorem}
\label{thm:convergence}
\begin{enumerate}[i)]
\item
Given a smooth curvature line parametrized surface
$f:I_1 \times I_2 \to \R^3$, where $I_1,I_2 \subset \R$ are compact intervals,
there exists a one-parameter family $f_\varepsilon$ of 2D cyclidic nets converging to $f$.
The frames related by \eqref{eq:frame_evolution} converge to the corresponding smooth frame of $f$ as well.
\item
Given a smooth triply orthogonal coordinate system
$f:I_1 \times I_2 \times I_3 \to \R^3$, where $I_1,I_2,I_3 \subset \R$ are compact intervals,
there exists a one-parameter family $f_\varepsilon$ of 3D cyclidic nets converging to $f$.
The frames related by \eqref{eq:frame_evolution} converge to the corresponding smooth frame of $f$ as well.
\end{enumerate}
\end{theorem}

\subsection{Higher dimensional cyclidic nets}
\label{subsec:cyclidic_nets_md}

For an analogous geometric interpretation (related to cyclidic
patches in $\R^3$) of $m$D cyclidic nets in $\R^N, N \ge 4$, one
has to take the following into account: Also in $\R^N$ the frames
at vertices determine circular arcs which should serve as boundary
arcs of surface patches, but generically the four vertices of an
elementary quad and the four circular arcs connecting them are
\emph{not} contained in an affine 3-space. Nevertheless there
exists a unique surface patch for the given boundary, which is the
image of a cyclidic patch in $\R^3$ under a M\"obius
transformation: Due to the evolution \eqref{eq:frame_evolution}
the vertices and circular arcs associated to one patch are always
contained in a 3-sphere (cf. Fig.~\ref{fig:patch_in_4_space}). The
corresponding patch is obtained by identifying this sphere with
$\R^3$ via stereographic projection, so the $m$D case can be
reduced to the 3D case.

\begin{figure}[htb]
\centering
 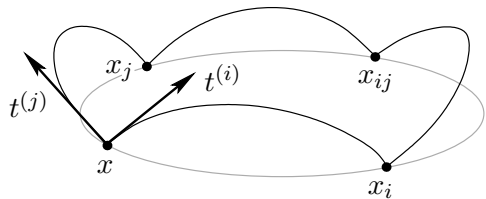 
\caption{In dimensions greater than 4,
the shown curves are generically contained in a 3-sphere, which is determined by the vertices
and the tangent plane at $x$ spanned by $t^{(i)}$ and $t^{(j)}$.}
\label{fig:patch_in_4_space}
\end{figure}

\subsection{Computer implementation}
\label{subsec:implementation}

The 3D pictures in this work were produced with a Java application based on the
projective description of Dupin cyclides in Lie geometry. This application is
available as a Java Webstart at the webpages of the DFG Research Unit
``Polyhedral
Surfaces''\footnote{\texttt{http://www3.math.tu-berlin.de/geometrie/ps/software.shtml}}.
The main tool for visualization is the open source class library
\texttt{jReality} \cite{jReality}, while the projective model of Lie geometry is
implemented as part of \texttt{jTEM} \cite{jTEM}.

For a given circular net (2D or 3D) in $\R^3$ one chooses an initial frame
which is then reflected to all vertices using \eqref{eq:frame_evolution}.
As explained in Section~\ref{subsec:cyclidic_nets_geometry}
one has a unique cyclidic patch for each elementary quadrilateral.
In the following we will explain how to choose parametrizations of the individual cyclidic patches,
such that for the induced global parametrization of the whole cyclidic net
all parameter lines are continuous.

To get a concrete parametrization of a single patch determined by $(X,B)$
one has to chose additional points $y^{(i)},y^{(j)}$ on boundary arcs $\what {x,x_i},\what {x,x_j}$.
Previously we chose $y$'s to be the midpoints (cf. Section~\ref{subsec:curvature_line_param_patches}).
But even if $\sigma^{(j)}$ touches $s^{(i)}$ in the midpoint of $\widehat{x,x_i}$,
in general $\sigma^{(j)}$ does not touch $s_j^{(i)}$ in the midpoint of $\widehat{x_j,x_{ij}}$.
So the idea is to choose points $y^{(i)}$ along two intersecting discrete curvature lines as midpoints of boundary arcs,
and then to keep track of the intersection points of $\frac{1}{2}$-parameter lines with opposite boundary arcs.
These new intersection points should then be used in \eqref{eq:curvature_sphere_from_frame} in order to obtain $\sigma$'s for adjacent patches.
One obtains parametrizations of the patches of a 2D cyclidic net,
such that all $\frac{1}{2}$-parameter lines are continuous
(cf. Fig.~\ref{fig:global_curvline}).

\begin{figure}[htb]
\centering
 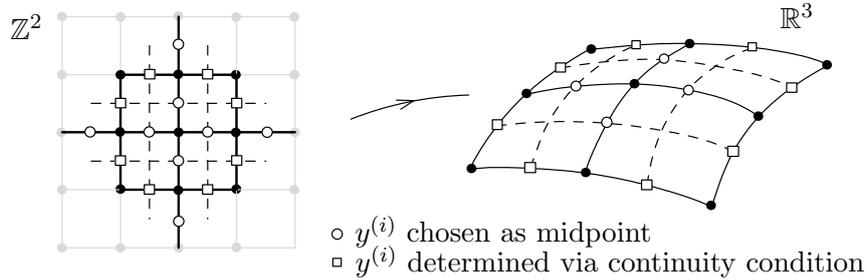 
\caption{Evolution of the points $y^{(i)}$ for initial $y^{(i)}$'s
associated to edges along two intersecting coordinate axes of $\Z^2$.}
\label{fig:global_curvline}
\end{figure}

To prove continuity of all parameter lines,
we verify continuity across the common boundary arc of adjacent patches:
Consider the induced pa\-ra\-me\-tri\-zations over $[0,1]$ of these common arcs.
Since parametrizations are quadratic an coincide at three points, they coincide identically.

In the case of 3D cyclidic nets
closeness conditions enter the game which have to be satisfied in order to obtain continuous parameter lines.
The reason is that $y$'s are associated to edges of $\Z^m$,
and there are different ways to determine the same $y$.
But indeed the evolution of $y$'s is consistent,
which follows from Theorem~\ref{thm:patch_in_between}
(cf. Fig.~\ref{fig:patch_in_between}).

\begin{appendix}

\section{Lie geometry of oriented spheres}
\label{sec:lie_geometry}

For the reader's convenience we include a brief overview of the projective model of Lie geometry.
Probably the most elaborate source is the classical book \cite{Blaschke:1929:DG3} by Blaschke,
a modern comprehensive introduction is contained in Cecil \cite{Cecil:1992:LieSphereGeometry}.
A short introduction can be found in \cite{BobenkoSuris:2008:DDGBook}.

\paragraph{Subject of Lie Geometry.}

Lie geometry is the geometry of oriented hyperspheres in
$\RI^N = \R^N \cup \{ \infty \} \cong \S^N.$
Points are considered as spheres of vanishing radius
while oriented hyperplanes are considered as spheres of infinite radius.
The transformation group of Lie geometry consists of so-called \emph{Lie transformations}.
They are bijections on the set of oriented spheres
and characterized by the preservation of \emph{oriented contact}
(cf. pp.~\pageref{lie_geo:oriented_contact}f).

M\"obius geometry as well as Laguerre geometry are subgeometries of Lie geometry.
M\"obius geometry studies properties which are invariant with respect to M\"obius transformations,
and accordingly Laguerre geometry studies properties which are invariant with respect to Laguerre transformations.
M\"obius transformations are exactly those Lie transformations which preserve the set of point spheres,
and Laguerre transformations are exactly those Lie transformations which preserve the set of hyperplanes.

\paragraph{The Lie quadric.}
\label{lie_geo:lie_quadric}

The space of homogeneous coordinates is $\R^{N+1,2}$,
i.e. $\R^{N+3}$ equipped with an inner product
$\langle \cdot , \cdot \rangle$ of signature $(N+1,2)$.
On the standard basis $(\e_1,\dots,\e_{N+3})$
this product acts as
\begin{equation*}
\langle \e_i , \e_j \rangle =
\left\{
\begin{array}[]{rl}
1, & i = j = 1,\dots,N+1,\\
-1, & i = j =  N+2,N+3,\\
0, & i \ne j.
\end{array}
\right.
\end{equation*}
Homogeneous coordinates are marked with a hat,
except for basis vectors which are written in bold face.
Points in a projective space we write as $x = [\what x]$.

\begin{definition}[Lie quadric]
We denote the set of isotropic vectors in $\R^{N+1,2}$ by
\[ \LL^{N+1,2} := \left\{ \what v \in \R^{N+1,2} \mid \langle \what v, \what v \rangle = 0 \right\}.\]
The \emph{Lie quadric} is the projectivation of $\LL^{N+1,2}$,
\[ \mathcal{Q}^{N+1,2} := \P(\LL^{N+1,2}) \subset \RP^{N+1,2} = \P(\R^{N+1,2}). \]
\end{definition}

For a convenient description of the Lie quadric we introduce the vectors
\begin{equation*}
\e_0 := \frac12 \left( \e_{N+2} - \e_{N+1} \right),\quad
\e_\infty := \frac12 \left( \e_{N+2} + \e_{N+1} \right),\quad
\e_r := \e_{N+3}
\end{equation*}
and replace the basis $(\e_1,\dots,\e_{N+3})$
by $(\e_1,\dots,\e_N,\e_0,\e_\infty,\e_{r})$.
Note that the vectors $\e_0,\e_\infty$ are isotropic, i.e. $\e_0,\e_\infty \in \LL^{N+1,2}$,
and $\langle \e_0 , \e_\infty \rangle = -\frac{1}{2}$.

Generalized oriented Hyperspheres in $\RI^N$ are in bijection with points on the Lie quadric:

\begin{enumerate}[$\quad \bullet$]
\item
    Normalized homogeneous coordinates of a proper oriented hypersphere $s$
    with center $c \in \R^N$ and oriented radius $r \in \R\setminus \left\{ 0 \right\}$ are
\begin{equation}
\what s = c + \e_0 + \left( \|c\|^2 - r^2 \right)\e_\infty + r \cdot \e_{r}.
\label{eq:representative_sphere}
\end{equation}
    We follow the convention that positive radii $r>0$ are assigned to spheres with the inward field of unit normals.

\item Points $x \in \R^N$ are considered as hyperspheres of vanishing radius,
\begin{equation}
\what x = x + \e_0 + \|x\|^2 \e_\infty + 0 \cdot \e_{r}.
\label{eq:representative_point}
\end{equation}
$[\what v] \in \mathcal Q^{N+1,2}$ is a point sphere if and only if it is contained the projective hyperplane $\P(\e_r^\perp)$,
i.e. $\langle \what v,\e_r \rangle = 0$.

\item The point $\infty$ is also a point sphere and has homogeneous coordinates
    \begin{equation}
     \what \infty = \e_\infty.
    \end{equation}
$[\e_\infty]$ is the only isotropic point with vanishing $\e_0$- and $\e_r$-components.

\item
Normalized homogeneous coordinates of an oriented hyperplane $p$
with normal $n \in \S^{N-1}$ and offset $d \in \R$ are
\begin{equation}
\what p = n + 0 \cdot \e_0 + 2d \e_\infty + \e_{r}.
\label{eq:representative_plane}
\end{equation}
$[\what v] \in \mathcal Q^{N+1,2}$ is a hyperplane if and only if it is contained the projective hyperplane $\P(\e_\infty^\perp)$,
i.e. $\langle \what v,\e_\infty \rangle = 0$.
\end{enumerate}

\paragraph{Oriented contact and contact elements.}
\label{lie_geo:oriented_contact}

From the Euclidean viewpoint there are the following possible configurations of spheres in oriented contact:

\begin{enumerate}[$\quad \bullet$]
\item
Two proper hyperspheres are in oriented contact if they are tangent
with coinciding normals in the point of contact (cf. Fig.~\ref{fig:oriented_contact_spheres}).
\begin{figure}[htb]
\centering
 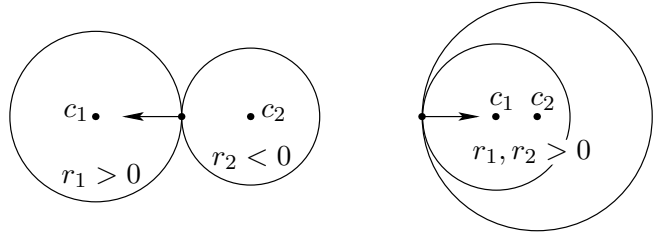 
\caption{Oriented contact of two proper hyperspheres.}
\label{fig:oriented_contact_spheres}
\end{figure}
\item
A proper hypersphere and a hyperplane are in oriented contact if they are tangent with coinciding normals in the touching point.
\item
Two hyperplanes are in oriented contact if they are parallel with coinciding normal,
their point of contact is $\infty$.
\item
Oriented contact with a point means incidence.
A point cannot be in oriented contact with any other point.
\end{enumerate}
Easy to check (see e.g. \cite{BobenkoSuris:2008:DDGBook}) is the crucial fact
\begin{quote}
\emph{
Two generalized oriented hyperspheres $s_1$ and $s_2$ are in oriented contact
if and only if $\langle \what s_1 , \what s_2 \rangle = 0$,
i.e. if $s_1,s_2 \in \mathcal Q^{N+1,2}$ are polar with respect to the Lie quadric.
}
\end{quote}

Given two spheres in oriented contact, there is a unique point of
contact. At most one of the spheres is a point sphere, thus we
have a unique normal $n$ at the contact point $x$. We identify
this pair $(x,n) \in \RI^N \times \S^{N-1}$ with the 1-parameter
family of spheres through $x$ with normal $n$ at $x$, all being in
oriented contact. Such a sphere pencil is called a \emph{contact
element} (cf. Fig.~\ref{fig:contact_element}). For $x=\infty$ a
contact element consists of parallel hyperplanes with coinciding
orientation.
\begin{figure}[htb]
\centering
 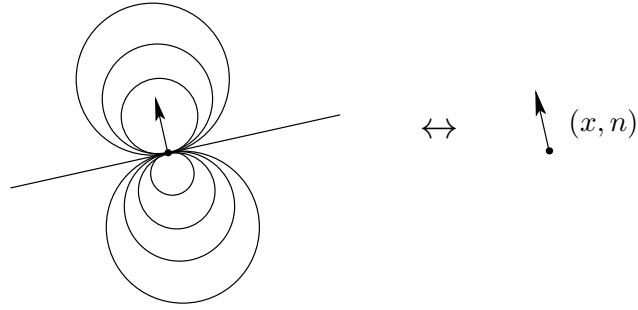 
\caption{A contact element is composed of all hyperspheres touching in a point $x$ with coinciding normal $n$.}
\label{fig:contact_element}
\end{figure}

The projective translation gives

\begin{definition}[Contact element]
A \emph{contact element} of $\; \RI^N$ is an isotropic line in $\RP^{N+1,2}$,
i.e. a line contained in the Lie quadric $\mathcal{Q}^{N+1,2}$.
We denote the set of all these lines by
\begin{equation*}
\L^{N+1,2}_0 := \left\{ \text{Isotropic lines in } \RP^{N+1,2} \right\}.
\end{equation*}
\label{def:contact_elements}
\end{definition}

A quadric of signature $(p,q)$ in a projective space contains only projective subspaces $U$
of dimension $\dim U < \min(p,q)$.
Thus the Lie quadric contains isotropic lines, i.e. contact elements, but no isotropic planes.
Equivalently there are no three coplanar isotropic lines in $\mathcal Q^{N+1,2}$.
This means that if three oriented spheres are pairwise in oriented contact,
they have to belong to one contact element.

Another important implication is the following:
\label{unique_polar_sphere}
For a line $L \subset \mathcal{Q}^{N+1,2}$ and a point $s \in \mathcal{Q}^{N+1,2}, s \notin L$,
there is a unique point $\sigma \in L$ such that $s$ and $\sigma$ are polar,
i.e. there is a unique isotropic line $L_s \ni s$ which intersects $L$ (in $\sigma$).
In other words, for each contact element $L$ and a sphere $s$ not contained in $L$,
there is a unique sphere $\sigma \in L$ such that $s$ and $\sigma$ are in oriented contact.

Spheres $(s_1,s_2)$ in oriented contact span a contact element $\inc[s_1,s_2]$ and each contact element contains exactly one point sphere.
Homogeneous coordinates of point spheres are characterized by vanishing $\e_r$-component,
thus the contact point of a contact element $L$ is the unique intersection $L \cap \P(\e_r^\perp)$.
Solving $(\alpha_1 r_1 + \alpha_2 r_2) \e_r = 0$ and normalizing the $\e_0$-component yields

\begin{lemma}
\label{lem:hom_coords_of_contact_points}
Normalized homogeneous coordinates \eqref{eq:representative_point} of the unique point sphere
contained in a contact element $\inc[s_1,s_2]$ are given by
\begin{equation}
\label{eq:hom_coords_of_contact_points}
\what x = \frac 1 {r_2 - r_1} \left( r_2 \what s_1 - r_1 \what s_2 \right).
\end{equation}
In case that one of the spheres is a hyperplane,
i.e. the contact element is given as $\inc[s,p]$,
this formula reduces to
\begin{equation*}
\what x = \what s - r \what p.
\end{equation*}
\end{lemma}

\paragraph{Lie sphere transformations.}

Lie sphere transformations of $\RI^N$ map spheres to spheres and are
characterized by the fact that they preserve contact elements.
In the projective model they are described as projective transformations of $\RP^{N+1,2}$ which preserve the Lie quadric and thus map isotropic lines to isotropic lines.

\paragraph{Description of circles.}

The description of points on a circle in the projective model of Lie geometry is essential for us.

\begin{lemma}[]
Four points in $\RI^N$ lie on a circle,
if and only if the corresponding points in
$\mathcal Q^{N+1,2} \cap \textup{P}(\e_r^\perp) \subset \RP^{N+1,2}$ are coplanar.
\label{lem:circle_planar_representatives}
\end{lemma}
The corresponding M\"obius geometric claim is a classical result,
see for example \cite[Thm. 3.9]{BobenkoSuris:2008:DDGBook}.
Lemma~\ref{lem:circle_planar_representatives} follows immediately since M\"obius geometry
corresponds to the restriction of Lie quadric 
$\mathcal Q^{N+1,2}$ to the hyperplane $\textup{P}(\e_r^\perp)$ of point spheres.

We call a continuous family of contact elements sharing one sphere an
\emph{isotropic cone}.  Such a cone describes a continuous curve on the common
sphere (or plane) corresponding to its tip.
Lemma~\ref{lem:circle_planar_representatives} implies that this curve is a
circle if and only if the isotropic cone is contained in a 3-dimensional
subspace of $\RP^{N+1,2}$.  Actually circles on a sphere $s$ are in bijection
with planes $\Pi \subset \P(\e_r^\perp) \cap \pol[s]$.  These planes in turn are
in bijection with 3-spaces $V \subset \pol[s]$ since $s \notin \Pi$, i.e. $V =
\inc[s,\Pi]$.  As each contact element $(x,n)$ containing $s$ may be written as
$\inc[s,x]$, one has $\left\{ \inc[s,x] \mid x \in \Pi \right\} = V \cap
\mathcal{Q}^{N+1,2} = C$.  In particular, $C$ is an isotropic cone.  This proves

\begin{proposition}
Circles on a sphere $s$ are in bijection with 3-dimensional subspaces $V \subset \emph{pol}[s]$ containing $s$.
The contact elements along such a circle are the generators of the corresponding isotropic cone
$C = V \cap \mathcal{Q}^{N+1,2}$.
\label{prop:contact_element_cone_circle}
\end{proposition}

\end{appendix}

\section*{Acknowledgement}

We want to thank Boris Springborn for numerous helpful discussions. We also
appreciate the help of Charles Gunn with implementing the Java application for
visualizing and exploring cyclidic nets, in particular by providing us with an
implementation of the projective model of Lie geometry.  Finally, we want to
thank Ulrich Bauer for his helpful comments.

\bibliographystyle{amsalpha}
\bibliography{huhnen-venedey}

\end{document}